\documentclass[reqno,a4paper]{amsart}
\usepackage{macros}
\usepackage{comment}

\title[Unipotent morphisms]{Unipotent morphisms}

\date{January 22, 2025}

\numberwithin{Theorem}{section} 
\newcommand{\Orb}{\mathcal{O}}
\newcommand{\Vect}{\mathbf{V}}
\newcommand{\QCoh}{\mathbf{QCoh}}
\renewcommand{\GL}{\mathrm{GL}}
\renewcommand{\U}{\mathrm{U}}
\renewcommand{\B}{\mathrm{B}}
\newcommand{\Flag}{\mathbf{Flag}}
\newcommand{\FLAG}{\mathsf{Flag}}
\newcommand{\SCH}[1]{\mathbf{Sch}/{#1}}
\usepackage[inline]{enumitem}
\theoremstyle{plain}
\newtheorem*{theorem*}{Theorem}
\newtheorem{maintheorem}{Theorem}

\theoremstyle{remark}
\newtheorem{claim}{Claim}[Theorem]
\newtheorem*{claim*}{Claim}
\newcommand{\spref}[1]{\href{https://stacks.math.columbia.edu/tag/#1}{#1}}

\newcommand{\classB}{B}

\newcommand{\RDERF}{\mathbf{R}}

\DeclareMathOperator{\Frame}{Fr}
\setlist[enumerate]{font=\normalfont}
\usepackage[inline]{enumitem}

\author{Daniel Bragg}
\email{braggdan@berkeley.edu}
\address{University of Utah\\Department of Mathematics\\155 South 1400 East\\
Salt Lake City, UT 84112, USA}
\author{Jack Hall}
\email{jack.hall@unimelb.edu.au}
\address{School of Mathematics \& Statistics\\The University of Melbourne\\Parkville, VIC, 3010, Australia}
\author{Siddharth Mathur}
\email{siddharth.mathur@uc.cl}
\address{Departamento de Matem\'atica\\
Pontificia Universidad Católica de Chile\\
Santiago, Chile.}

\begin{document}
\maketitle
\begin{abstract}
  We introduce the theory of unipotent morphisms of algebraic stacks
  and prove a surprising local to global principle for a class of
  vector bundles. Two sample applications of our methods are the
  following:
  \begin{enumerate*}
  \item a unipotent analogue of Gabber's Theorem for torsion
    $\mathbf{G}_m$-gerbes and
  \item smooth Deligne--Mumford stacks with quasi-projective coarse
    spaces satisfy the resolution property in positive characteristic.
  \end{enumerate*}
  Our main tool is a descent result for flags, which we prove using
  results of Sch\"appi.
\end{abstract}

\section{Introduction}
In the early 1990s, Gabber answered a question of Grothendieck for
quasi-compact schemes admitting an ample line bundle: every
cohomological Brauer class on such a scheme is 
represented by an Azumaya algebra
\cite{MR611868,deJong_result-of-Gabber}. This result has since been reinterpreted in terms of the \emph{resolution property} for algebraic stacks (i.e., every quasi-coherent sheaf is a quotient of a direct sum of vector bundles) and gerbes \cite{MR1844577,MR2026412, MR2108211,2013arXiv1306.5418G,mathur2021resolution}.
In this language, Gabber's result can be restated as follows.
\begin{theorem*}[Gabber] \label{thm:gabber} Let $X$ be a quasi-compact
  scheme that admits an ample line bundle. Let $\ms X \to X$ be a
  $\mathbf{G}_m$-gerbe. If the cohomology class $[\ms X]\in\H^2(X,\mathbf{G}_m)$ of $\ms X$ is torsion, then $\ms X$ is a global quotient stack (see Definition \ref{def:globalquotient}) and has the resolution property.
\end{theorem*}
Using quasi-projective methods and Gabber's Theorem, Kresch and
Vistoli show that smooth, separated and generically tame
Deligne--Mumford stacks with quasi-projective coarse moduli spaces
have the resolution property \cite{MR2026412}. It has long been a challenge to produce non-trivial vector bundles without such hypotheses. 

In this paper, we introduce new methods to construct non-trivial
vector bundles on schemes, algebraic spaces and algebraic stacks. Our
key idea is to 
leverage the abundance of cohomology in unipotent settings. A sample application of our results is the following additive
analog of Gabber's result (also see Theorem \ref{T:DM-gerbes}).
\begin{maintheorem}\label{MT:ga-gerbes-res}
  Let $X$ be a quasi-compact algebraic stack with affine diagonal and
  the resolution property. If $\ms G \to X$ is a $\mathbf{G}_{a}$-gerbe,
  then $\ms G$ has the resolution property.
\end{maintheorem}
In particular, this implies that $\mathbf{G}_a$-gerbes over smooth separated schemes have the resolution property. Note that the analogous multiplicative statement is still unknown, even in dimension three. We can also extend Kresch--Vistoli's result to the generically wild setting.
\begin{maintheorem}\label{MT:dm-stack-res}
  Let $\ms X$ be a smooth and separated Deligne--Mumford
  stack of finite type over a field $k$. Let $\pi \colon \ms X \to X$ be the associated
  coarse moduli space. If $X$ is
  quasi-projective over $\spec k$, then $\ms X$ has the resolution
  property.
\end{maintheorem}
We expect the key new technical input for the above results to be of
general interest: it states that one may descend certain vector
bundles along flat morphisms after locally taking a direct sum (see Theorem
\ref{T:descent-flags} for a more general statement).
\begin{maintheorem} \label{MT:flagbasic} Let $\ms X \to X$ be a
  morphism of quasi-compact and quasi-separated algebraic
  stacks. Suppose that $X$ has affine diagonal and satisfies the resolution property. If there
  is a faithfully flat morphism $X'\to X$, where $X'$ is affine, and a
  vector bundle $V'$ on $\ms X'= \ms X \times_X X'$ that admits a
  filtration
  \[
    0=V'_0 \subseteq V_1' \subseteq \dots \subseteq V_n'=V'
  \]
  with $V_{i}'/V_{i-1}'\simeq \Orb_{\ms X'}$ for all $i$ (i.e. $V'$ is \emph{trivially graded}), then there
  is a vector bundle $V$ on $\ms X$ and a split
  surjection $V_{\ms X'} \to V'$.
\end{maintheorem}

In fact, our results apply in the broader setting of \emph{unipotent
  morphisms}, a notion that we develop and characterize in
\S\ref{sec:unipotent-morphisms-groups}. Roughly speaking, the theory
is modeled on gerbes with unipotent stabilizers and algebraic stacks
of the form $[Z/\U_n]$, where $Z$ is an algebraic space and $\U_n$ is
the unipotent subgroup of $\GL_n$ consisting of unitriangular matrices. Compare with \cite{MR1844577,MR2108211,2013arXiv1306.5418G}, where quotients of the form $[Z/\GL_n]$ and the resolution property are considered. In fact, we present a unipotent enrichment of their results (see Theorem \ref{T:char-Runi}) and, better still, we prove a striking descent statement. Indeed, Theorem \ref{MT:flagbasic} reveals the following local to global principle: \emph{a locally unipotent morphism over a base with enough flags is globally unipotent} (see Theorem \ref{T:flat-local-unipotent-morphism}). 

These notions are closely related to many foundational questions of
independent interest. For example, over a field $k$ a \emph{unipotent
  group scheme} is a closed subgroup of $\U_{n,k}$ (see
\cite[Thm.~XVII.3.5]{SGA3-III-NEW} and
Remark \ref{rem:overafieldallagree} for several other
characterizations). It is natural to ask: to what extent is this true
over a general base?

\begin{Question} \label{Q1} Let $G \to S$ be a group scheme with
  unipotent geometric fibers. Is there an embedding $G \to \U_{n,S}$ for
  some $n \geq 1$?  Is this true locally on $S$?
\end{Question} 

As a consequence of our methods, if $S$ is a regular separated scheme or admits an ample line bundle,
then the existence of a flat-local embedding in $\U_n$ implies the existence of a Zariski-local embedding of $G$ in $\U_n$ (see Corollary \ref{C:flat-local-strong-unipotent} for a more general statement). In Section \ref{ss:geounivslocal}, we refine Question \ref{Q1} and explain some  cases where the answer is positive (see Examples \ref{E:char0-affine} and \ref{E:examples-r-unipotence}) and one case, in positive characteristic, where it is negative (Example \ref{E:non-separated-unipotent}(1)). One may also
compare Question \ref{Q1} with a question of Conrad on when
smooth affine group schemes can be embedded into $\GL_n$
\cite{mathoverflow_groups-over-dual-numbers}.

Unipotent morphisms are closely related to the notion
of a faithful moduli space, due to Alper \cite{Jarod-faithful}. This is a morphism of algebraic
stacks $f\colon  \ms X \to X$ such that
\begin{enumerate}
\item the natural map $\mathcal{O}_X \to f_*\mathcal{O}_{\ms X}$ is an isomorphism (i.e. $f$ is Stein); and 
\item if $F$ is quasi-coherent on $\ms X$ and $f_*F=0$, then $F=0$.
\end{enumerate}
Faithful moduli spaces should help with the analysis of moduli stacks
with only unipotent stabilizers (compare with good or adequate moduli
spaces in \cite{2008arXiv0804.2242A,MR3272912}) that appear in the
unstable locus of GIT quotients. In contrast to good moduli spaces,
however, faithful moduli spaces are rarely compatible with base
change. This makes them difficult to reduce to simple cases where we
can apply deformation theory. Faithful moduli spaces are also stable
under composition.

We show that faithful moduli spaces and locally unipotent morphisms
have unipotent stabilizers (Propositions
\ref{P:stabilizer-geom-unipotent} and \ref{P:res-gerb-faithful}) and
that faithful moduli spaces are generically gerbes for a unipotent
group (Proposition \ref{P:faithful-generic-gerbe}). We also show that
certain Stein unipotent morphisms are faithful moduli spaces
(Proposition \ref{prop:pushforward is faithful on objects}). 

\subsection*{Acknowledgments} We would like to thank Jarod Alper, Brian Conrad, Andrea Di Lorenzo, Aise Johan de Jong, Andrew Kresch, Nikolas Kuhn, Max Lieblich, Martin Olsson, Stefan Schr\"oer, Tuomas Tajakka, and Angelo Vistoli for helpful comments. We are especially thankful to David Rydh for carefully reading an earlier draft and providing many invaluable suggestions. We would also like to thank the anonymous referees for a number of helpful suggestions. The first author was supported by the National Science Foundation under
NSF Postdoctoral Research Fellowship DMS-1902875. The second author was partially supported by the Australian Research Council DP210103397 and FT210100405. The third author conducted part of this research in the framework of the research training group GRK 2240: Algebro-geometric Methods in Algebra, Arithmetic and Topology, which is funded by the Deutsche Forschungsgemeinschaft. The third author was also supported by the Swedish Research Council under grant no. 2016-06596 while the author was in residence at Institut Mittag-Leffler in Djursholm, Sweden during the Fall of 2021. The third author was also supported
by FONDECYT Regular grant No. 1230402.

\section{Vector bundles, global quotients, and the resolution property}
We briefly recall a circle of ideas from
\cite{MR1844577,MR2108211,2013arXiv1306.5418G}. We also make some
minor refinements, which will be useful in the present article. To agree with the conventions of \cite[\S 5]{2013arXiv1306.5418G}, a \emph{vector bundle} on $X$ will denote a locally free $\Orb_X$-module of finite constant rank.

Let $X$ be an algebraic stack and let $n$ be a non-negative
integer. If $V$ is a vector bundle on $X$ of rank $n$, there is an
associated \emph{frame bundle}
$\Frame(V) = \underline{\Isom}_X(\Orb_X^{\oplus n},V)$.  Then
$\Frame(V) \to X$ is a $\GL_n$-torsor and sits in the following
$2$-cartesian diagram:
\[
  \xymatrix{\Frame(V) \ar[r] \ar[d] & \spec \Z\ar[d] \\ X \ar[r] & \classB \GL_n. }
\]
It is well-known that this association induces an equivalence of
categories between vector bundles on $X$ of rank $n$ and
$\GL_n$-torsors. 
\begin{Definition} \label{def:globalquotient}
  Let $X$ be an algebraic stack. We say that $X$ is a \emph{global
    quotient} if $X \cong [Z/\GL_n]$, where $Z$ is a quasi-compact and quasi-separated algebraic space
  with an action of $\GL_n$ for some non-negative integer $n$. If $V$
  denotes the vector bundle of rank $n$ associated to the
  $\GL_n$-torsor $Z \to X$, we say that $V$ is \emph{faithful}. A
  morphism of algebraic stacks $f\colon X \to S$ is a \emph{global
    quotient} if there is a vector bundle $V$ of rank $n$ on $X$ with
   frame bundle that is quasi-compact, quasi-separated, and representable over $S$. In this case, we say that $V$
  is \emph{$f$-faithful}.
\end{Definition}
\begin{remark}
  A vector bundle $V$ on $X$ is easily seen to be $f$-faithful for
  $f\colon X \to S$ if and only if the relative stabilizer groups of
  $f$ act faithfully on $V$ at every geometric point of $X$. In particular, the relative stabilizer groups are affine \cite[Cor.~VI${}_{\mathrm{B}}$.1.4.2]{SGA3-I-NEW}. Moreover, if $X \xrightarrow{f}S \xrightarrow{g} T$ are quasi-compact and quasi-separated morphisms of algebraic stacks and $V$ is $(g\circ f)$-faithful, then
  $V$ is $f$-faithful. Also, $X$ is a global quotient if and only if
  the morphism $X \to \spec \Z$ is a global quotient.
\end{remark}
\begin{remark} \label{rem:vbfiniterank}
  It is natural to ask if algebraic stacks such as
  $X_1=\coprod_{n\geq 1} \classB \GL_{n,\Z}$ or $X_2=\classB \underline{\Z}$, where $\underline{\Z}$ denotes the constant group scheme associated to the integers, should be considered global
  quotients. There is certainly a locally free sheaf on $X_1$ with a faithful
  action, but it has non-constant and unbounded rank; also, $X_2$ admits a vector bundle of rank $2$ with a faithful action whose frame is quasi-compact but not quasi-separated. We exclude these as we are unable to shed much light on them and also to adhere to the conventions of \cite[\S 5]{2013arXiv1306.5418G}.
\end{remark}
Let $\mathcal{A}$ be an abelian category.  Let $\Lambda$ be a set of
objects of $\mathcal{A}$. Recall that $\Lambda$ \emph{generates}
$\mathcal{A}$ (or $\Lambda$ is \emph{generating}) if the functor
\[
  \mathsf{H}_\Lambda \colon \mathcal{A} \to \mathbf{Ab} \colon x\mapsto \prod_{a \in \Lambda} \Hom_{\mathcal{A}}(a,x).
\]
is faithful. If $\mathcal{A}$ is closed under small coproducts, then
$\Lambda$ is generating if and only if every $x\in \mathcal{A}$ admits
a presentation of the form
$\bigoplus_{a\in \Lambda} (a^{\oplus I_a}) \twoheadrightarrow x$.
\begin{remark}\label{R:gens-proj-zref}
  Let $\mathcal{A}$ be an abelian category and $\Lambda$ a set of
  objects of $\mathcal{A}$.  Then $\mathsf{H}_\Lambda$ is
  \emph{conservative} (i.e., if $f \colon x \to x'$ in $\mathcal{A}$
  and $\mathsf{H}_\Lambda(f)$ is bijective, then $f$ is an isomorphism) if and
  only if $\Lambda$ generates $\mathcal{A}$. Note that while
  $\mathsf{H}_\Lambda$ conservative implies that it is
  \emph{zero-reflecting} (i.e., $\mathsf{H}_\Lambda(x) = 0$ implies
  that $x = 0$), the converse does not hold in general---it does when
  the $a \in \Lambda$ are all $\mathcal{A}$-projective. For a
  counterexample, consider the affine line with the doubled origin; then
  $\mathsf{H}_{\{\Orb\}}$ is zero-reflecting but not conservative.
\end{remark}

\begin{Example}\label{E:gens-stack}
  If $X$ is an algebraic stack, then the abelian category $\QCoh(X)$
  always admits a set of generators \cite[Tag
 \spref{0781}]{stacks-project}. If $X$ is quasi-compact and quasi-separated, then 
 a set of generators can always be found amongst the $\Orb_X$-modules of 
 finite presentation \cite{rydh_noetherian-general}. Note this result is much simpler
 if $X$ is noetherian \cite[Prop.~15.4]{LMB} or has quasi-finite and separated diagonal \cite{rydh-2009}.
\end{Example} 
The \emph{resolution property}, which we refine and recall below, is about algebraic stacks that have a generating set of vector bundles.
\begin{Definition} \label{D:res-prop} Let $X$ be an algebraic
	stack. Let $\Vect$ be a set of vector bundles on
	$X$. We say that $X$ has the \emph{$\Vect$-resolution property}
	if it is quasi-compact, quasi-separated, and $\Vect$ generates
	$\QCoh(X)$. Following \cite[Def. 5.1]{2013arXiv1306.5418G}, we say that $X$ has the \emph{resolution
		property} if it has the $\Vect$-resolution property for some set $\Vect$.
\end{Definition}

\begin{remark} \label{rem:noetherianRP} Fix a noetherian algebraic stack $X$ and let $\Vect$ denote the set of isomorphism classes of vector bundles on $X$. Then $X$ has the $\Vect$-resolution property if and only if every coherent sheaf on $X$ is the quotient of a locally free sheaf of finite rank (see \cite[Rem. 5.2]{2013arXiv1306.5418G}). \end{remark}

Just as in \cite{2013arXiv1306.5418G}, we will be interested in a relative version of the resolution property, so we recall the following definition.
\begin{Definition}(\cite[Def.\ 2.7]{2013arXiv1306.5418G})
	Let $f \colon X \to S$ be a morphism of algebraic stacks. Let $\mathbf{V}$ be a set of finitely presented quasi-coherent sheaves on $X$. We say that $\mathbf{V}$ is \emph{$f$-generating} if $f$ is quasi-compact and quasi-separated and there is a set of (not necessarily finitely presented) quasi-coherent $\Orb_S$-modules $\mathbf{W}$ such that $\Vect \tensor_{\Orb_X} f^*\mathbf{W}=\{V \tensor_{\Orb_X} f^*W \,:\, V \in \Vect,\, W \in \mathbf{W}\}$ generates $\QCoh(X)$. We say that $\mathbf{V}$ is \emph{universally} $f$-generating if for every morphism of algebraic stacks $s \colon S' \to S$, the set $s'^*\Vect=\{s'^*V\colon V \in \mathbf{V}\}$, where $s' \colon X'=X\times_S S' \to X$ and $f' \colon X' \to S'$ denote the projections, is $f'$-generating. 
\end{Definition}
\begin{remark}\label{rem:universal-qaff}
 Note that $\Vect$ is universally $f$-generating if and only if it is so on an fpqc covering of the target (see, \cite[Prop.\ 2.8 (iii)]{2013arXiv1306.5418G}). Also, if $\Vect$ is $f$-generating, then it remains so after quasi-affine base change (combine \cite[Prop.\ 2.8 (ii),(v),(vi)]{2013arXiv1306.5418G}). In particular, if $S$ has quasi-affine diagonal, then $\mathbf{V}$ is $f$-generating if and only if it is universally $f$-generating (see \cite[Cor.\ 2.10]{2013arXiv1306.5418G}). Ignoring set-theoretic issues,  Remark \ref{E:gens-stack} implies that we can take $\mathbf{W} = \QCoh(S)$ in the definition of $f$-generating.
\end{remark}
 \begin{Definition}
 Let $f \colon X \to S$ be a morphism of algebraic stacks. If $\Vect$ is a set of vector bundles on $X$, we say that $f$ has the \emph{$\Vect$-resolution property} if $\Vect$ is universally $f$-generating. Following \cite[Def.\ 5.1]{2013arXiv1306.5418G}, we say that $f$ has the \emph{resolution property} if $f$ has the $\Vect$-resolution property for some $\Vect$. If such a $\mathbf{V}$ arises as a set of submodules of the
	polynomials in a fixed vector bundle $V$ and its dual that are split
	by restriction to the frame bundle of $V$, then we say that $V$ is an
	\emph{$f$-tensor generator} \cite[Defn.~6.1]{2013arXiv1306.5418G}. Such a $V$ is \emph{$f$-R-faithful} if it is also $f$-faithful (note that $f$-tensor-generators are automatically $f$-faithful if $f$ has relatively affine stabilizers by \cite[Thm.\ 6.4]{2013arXiv1306.5418G}). 
\end{Definition}

\begin{remark} \label{rem:composition} An algebraic stack $X$ has the $\Vect$-resolution property if and
only if the morphism $X \to \Spec \mathbf{Z}$ has the $\Vect$-resolution
property. In general, if $Y$ has the $\mathbf{V}'$-resolution
property and $f \colon X \to Y$ is a morphism of algebraic
stacks with the $\Vect$-resolution property, then $X$ has the $\Vect\otimes_{\Orb_X} f^*\mathbf{V}'$-resolution property \cite[Prop.\ 2.8(v)]{2013arXiv1306.5418G}.
\end{remark}
\begin{Example}
    Let $A$ be an abelian scheme of positive dimension over a field $k$. If $g\colon \spec k \to BA$ denotes the standard covering, then $\Orb_{\spec k}$ is $g$-generating but not universally $g$-generating. Also, $g^*$ induces an equivalence $\QCoh(BA) \simeq \QCoh(\spec k)$, so $f\colon BA \to \Spec k$ has the resolution property with $\mathcal{O}_{BA}$ as a $f$-tensor-generator. However, $BA$ admits no $f$-faithful vector bundles.
\end{Example}
\begin{Example} \label{ex:knownres} If $X$ admits an ample family
  $L_1$, $\dots$, $L_r$ of line bundles \cite[Tag \spref{0FXR}]{stacks-project} (e.g. $X$ is quasi-projective over a
  field or is noetherian, normal and $\Q$-factorial with affine diagonal), then it satisfies the
  $\Vect$-resolution property \cite[Ex.\ 5.9(i)]{2013arXiv1306.5418G},
  where $\Vect = \{ L_i^{m} \,:\, m\leq 0,\, i=1,\dots, r\}$. In
  dimensions two and above, there are proper schemes with no
  nontrivial line bundles.  However, two-dimensional separated
  algebraic spaces always have enough vector bundles (see
  \cite[Thm.\ 2.1]{NormalSurfaceRes}, \cite[Thm.\ 5.2]{gross_2012}, and
  \cite[Thm.\ 41]{mathur2021resolution}). Nothing is known for smooth separated
  algebraic spaces or normal separated schemes over a field in dimension $\geq 3$, except that they enjoy the resolution property after removing a closed subset of codimension $\geq 3$ \cite{zbMATH07641743}.
\end{Example}
\begin{Example} \label{ex:stackres} For algebraic stacks the situation
  is more complicated. Kresch and Vistoli showed that a
  Deligne--Mumford stack which is smooth, generically tame and separated over a
  field satisfies the resolution property if its coarse moduli space
  is quasi-projective \cite[Thm.~1.3]{MR2026412}. The case of $2$-dimensional normal
  tame algebraic stacks with finite diagonal over a field was settled in
  \cite[Thm.~1]{mathur2021resolution}. There are also non-torsion (and
  hence, non-regular) $\mathbf{G}_m$-gerbes which do not have the
  resolution property \cite[Rem.~1.11b]{Brauer2}. As far as the authors are aware, there
  is no known example of a separated or smooth algebraic stack with affine
  diagonal which does not have the resolution property.
\end{Example}

Surprisingly, Totaro and Gross (see \cite[Thm.\ 1.1]{MR2108211} and
\cite[Thm.\ 1.1, Thm.\ 6.10]{2013arXiv1306.5418G}) were able to relate
quotient stacks and the resolution property. We have the following
refinement, whose proof is the same as Gross', but we just keep track of
the generating set.
\begin{Theorem} \label{thm:totgross} Let $f \colon X \to S$ be a
  morphism of quasi-compact and quasi-separated algebraic
  stacks. Assume that $f$ has affine stabilizers.
  \begin{enumerate}
  \item  \label{thm:totgross:gens} Let $\Vect$ be a set of vector bundles on $X$ that is closed
    under direct sums. If $\Vect$ is universally $f$-generating, then there is a
    vector bundle $V\in \Vect$ of rank $n$ on $X$ whose frame bundle
    is quasi-affine over $S$. 
  \item \label{thm:totgross:frame} Let $V$ be a vector bundle on
    $X$. Then $V$ has frame bundle quasi-affine over $S$ if and only
    if $V$ is an $f$-tensor generator.
  \end{enumerate}
In particular, a quasi-compact and quasi-separated
    algebraic stack with affine stabilizers at closed points has the resolution
    property if and only if it can be written in the form $[U/\GL_n]$,
    where $U$ is a quasi-affine scheme.
\end{Theorem}
\begin{proof}
  Claim \eqref{thm:totgross:frame} is \cite[Thm.\
  6.4]{2013arXiv1306.5418G}. For \eqref{thm:totgross:gens}: consider the
  natural inverse system of $X$-stacks formed by
  \[
    F_J\colon \prod_{W \in J} \Fr(W) \to X,
  \]
  where $J \subset \Vect$ is a finite subset. Since each frame bundle
  is affine over $X$, the transition maps in the inverse system are
  affine. Hence, the inverse limit is an algebraic stack
  $p\colon F \to X$, which is affine over $X$ \cite[Thm.\
  C(i)]{rydh-2009}. Since $\Vect$ is generating for $f$ and every
  vector bundle in $\Vect$ becomes trivial when restricted to $F$, it
  follows that $\Orb_F$ is a tensor generator for the morphism
  $f \circ p \colon F \to X \to S$. Thus, the morphism $f\circ p$ is
  quasi-affine \cite[Prop.\ 4.1]{2013arXiv1306.5418G}. By
  absolute approximation \cite[Thm.\ C(i)]{rydh-2009}, there is a finite subset
  $J \subset \Vect$ such that $F_J \to S$ is quasi-affine. By
  \cite[Lem.\ 1.1]{rydh2009remarks}, the frame bundle of
  $V=\bigoplus_{W \in J} W \in \Vect$ is quasi-affine. For the final claim, it suffices to show that an algebraic stack with the resolution property and affine stabilizers at closed points, has affine stabilizers everywhere. This is just \cite[Lem.~5.15]{2013arXiv1306.5418G}.
\end{proof}

If a quasi-compact and quasi-separated scheme has an ample line
bundle, then it can be written in the form $[U/\GL_1]$, where $U$ is a
quasi-affine scheme. From this perspective, it is natural to view the
resolution property as a higher-dimensional analogue of
quasi-projectivity. Note that the converse does not hold (e.g., take
$U=\A^2-\{0\}$ with $\GL_1$ acting with weights $(1,-1)$; then
$[U/\GL_1]$ is the line with the doubled-origin, which is not
separated, so does not admit an ample line bundle---but it does admit an ample family of line bundles).

\begin{definition}\label{D:embedding-groups}
  Let $S$ be an algebraic stack. We say that a morphism $G \to S$ is a
  \emph{group} if it is a group object in representable morphisms over $S$.  If $G \to S$ is a group that is flat and of finite presentation, then we say that it is
  \begin{enumerate}
  \item \emph{embeddable} if it admits a group monomorphism to $\GL(E)$ for
    some vector bundle $E$ of rank $n$ on $S$; 
      \item \emph{R-embeddable}  if it admits a group monomorphism to $\GL(E)$
    for some vector bundle $E$ of rank $n$ on $S$ such that the quotient
    $\GL(E)/G$ is quasi-affine over $S$ (note that this implies that
    $G \to \GL(E)$ is a closed immersion).
  \end{enumerate}
\end{definition}
\begin{Example} \label{ex:faithful} An embedded group is
  not necessarily R-embedded, even over a field. Indeed, if 
  $\B_{2,\C} \subset \GL_{2,\C}$ is a Borel subgroup, then
  $\GL_{2,\C}/\B_{2,\C} \simeq \mathbf{P}^1_\C$. 
\end{Example}
\begin{remark}
  If $G \to S$ is a group that is flat and of finite presentation, then $G \to S$ is embeddable if and only if the structure morphism of
  the classifying stack $\classB G \to S$ is a global
  quotient. Indeed, if $\classB G \to S$ is a global quotient, then
  there is a vector bundle $E$ on $\classB G$ whose total space is
  representable over $S$. This gives a representable $S$-morphism
  $\classB G \to \classB \GL_{\rank E}$; the induced map on inertia
  groups gives a monomorphism $G \hookrightarrow \GL(E')$ for some
  vector bundle $E'$ on $S$. The converse is similar. In particular,
  $G \to S$ is embeddable if and only if there is a vector bundle $E$
  on $S$ with a faithful action of $G$.  We call these \emph{faithful}
  $G$-representations. If, in addition, $G \to S$ has affine fibers, then it is R-embeddable if and only if
  $\classB G \to S$ has the resolution property (Theorem \ref{thm:totgross}).  In particular,
 such a $G \to S$ is R-embeddable if and only if there is a vector bundle
  $E$ on $S$ with a faithful action such that the quotient
  $\GL(E)/G \to S$ is a quasi-affine morphism. We will call these
  $G$-representations \emph{R-faithful}.
\end{remark}
We expect the following result to be well-known to experts. Over a field, it is due to
Rosenlicht \cite[Thm.~3]{MR0130878}.
\begin{proposition}
  \label{P:hom} Let $S$ be an algebraic stack.
  \begin{enumerate}
  \item \label{lem:i:hom:upper} The quotient morphism
    $\GL_{n,S}/\U_{n,S} \to S$ is quasi-affine. 
  \end{enumerate}
  Let $G \subseteq \U_{n,S}$ be a closed subgroup that is flat and of
  finite presentation over $S$.
  \begin{enumerate}[resume]
  \item \label{lem:i:hom:sub} If $S$ is a normal scheme
     and the quotient $\U_{n,S}/G$ is representable by a scheme (e.g. if $S$ is Dedekind or the spectrum of a field, see \cite[Rem.~VI${}_{\mathrm{B}}$.9.3(b)]{SGA3-I-NEW}), then the quotient $\U_{n,S}/G \to S$ is quasi-affine. 
  \item \label{lem:i:hom:finite} If $G \to S$ is finite, then the
    quotient $\U_{n,S}/G \to S$ is affine.
  \end{enumerate}
\end{proposition}
\begin{proof}
  Since the inclusion $\U_{n,S} \subset \GL_{n,S}$ is defined over
  $\Spec \mathbf{Z}$, in \eqref{lem:i:hom:upper} it suffices to assume
  $S=\Spec \mathbf{Z}$. In particular, we may assume that $S$ is
  Dedekind. We will now prove that both of the homogenous spaces
  $\U_{n,S}/G$ and $\GL_{n,S}/\U_{n,S}$ are quasi-affine over $S$
  using \cite[Thm.~VII.2.1]{Raynaud}, thereby establishing both
  \eqref{lem:i:hom:upper} and \eqref{lem:i:hom:sub}. Since $\GL_{n,S}$
  and $\U_{n,S}$ are both smooth with connected fibers, it suffices to
  show that the abelian groups $\Pic(\U_{n,k(s)}/G_{k(s)})$ and
  $\Pic(\GL_{n,k(s)}/\U_{n,k(s)})$ are torsion. To this end, first
  note that the cohomology of the complex
  \[
    \Hom(K,\mathbf{G}_m) \to \Pic(H/K) \to \Pic(H)
  \]
  is torsion for any inclusion of finite-type group schemes
  $K \subset H$ over a field by
  \cite[Thm.~VII.1.5]{Raynaud}. Moreover, by
  \cite[Prop.~XVII.2.4(ii)]{SGA3-II-NEW} we always have
  $\Hom(K,\mathbf{G}_m)=0$ for every unipotent group scheme $K$ over a
  field. Further, the Picard groups of $\U_{n,k(s)}$ and
  $\GL_{n,\kappa(s)}$ both vanish. Indeed, in the former case the
  underlying scheme is isomorphic to affine space, and in the latter case
  it is a principal open of affine space. Thus, the two outer terms in
  the complex are zero in both of our cases and it follows that
  $\Pic(\U_{n,k(s)}/G_{k(s)})$ and
  $\Pic(\GL_{n,k(s)}/\U_{n,k(s)})$ are torsion, as desired.

  For \eqref{lem:i:hom:finite}: the morphism $\U_{n,S} \to \U_{n,S}/G$
  is finite flat. Since $\U_{n,S} \to S$ is affine, the same is true
  for $\U_{n,S}/G$ (e.g.~\cite[Thm. 8.1]{rydh-2009}).
\end{proof}

\begin{proposition} \label{prop:framebundle} Let $f \colon X \to S$ be a quasi-compact and quasi-separated morphism of algebraic stacks. 

\begin{enumerate} 

\item \label{prop:framebundle:sum} Let $W$ and $V$ be vector bundles on $X$. If the frame bundle $F_W$ is representable (resp., quasi-affine) over $S$, then the frame bundle of $V \oplus W$ is also representable (resp.,  quasi-affine) over $S$.
\item \label{prop:framebundle:cohaff} If $f$ is adequately affine
  (e.g. a gerbe banded by a reductive group scheme, a good/adequate
  moduli space morphism, or a coarse moduli space morphism \cite[Def.~4.1.1]{MR3272912}), then
  every $f$-faithful vector bundle $V$ on $X$ has a frame bundle
  which is affine over $S$.
\end{enumerate}  \end{proposition} 

\begin{proof}
  For the representability part of \eqref{prop:framebundle:sum}: the
  hypothesis implies that the relative stabilizers act faithfully on the
  fibers of $W$ and therefore they must act faithfully on the fibers
  of $V \oplus W$, as desired. For the latter claim, see
  \cite[Lem.~1.1]{rydh2009remarks}. For
  \eqref{prop:framebundle:cohaff}: $F_V \to S$ is representable and
  adequately affine, so affine \cite[Thm.~4.3.1]{MR3272912}.
\end{proof}
\section{Sch\"appi's Theorem}
Let $A$ be a ring and let $M$ be a flat $A$-module. Lazard's theorem
\cite{MR0168625} tells us that $M$ is a filtered colimit of free
$A$-modules of finite rank. Extending this result to the non-affine
situation is surprisingly subtle \cite{MR3079799}. Nonetheless,
Sch\"appi recently proved a remarkable Lazard-type theorem for a restricted
but very useful class of flat modules that arise in algebraic geometry
\cite[Thm.~1.3.1]{2012arXiv1206.2764S}.

Sch\"appi approaches his result through comodules over flat Hopf
algebroids. Since Sch\"appi's result is crucial for our article and
ought to be better known amongst algebraic geometers, we have
translated his category-theoretic proof into a direct proof for
algebraic stacks. We have made some simple but algebro-geometrically
natural generalizations to his hypotheses. Note that while non-affine
schemes tend to not have interesting projective objects, there are
interesting algebraic stacks (e.g., those with affine tame
\cite{MR2427954} or good \cite{2008arXiv0804.2242A} moduli spaces)
that do.
\begin{Theorem}[Sch\"appi]\label{T:schappi}
  Let ${f \colon Y \to X}$ be a flat morphism of quasi-compact and quasi-separated algebraic stacks. Let
  $\Vect$ be a set of isomorphism classes of vector bundles on $X$. Let $M$ be a vector
  bundle on $Y$ that is a projective object of $\QCoh(Y)$. If $X$ has the $\Vect$-resolution property, then $f_*M$ is a
  filtered colimit of finite direct sums of objects of
  $\Vect$ and their duals. In particular, if $Y$ is cohomologically affine
   (e.g., an affine scheme), then $f_*\Orb_Y$ is a
  filtered colimit of vector bundles.
\end{Theorem}
\begin{proof}
  We may replace $\Vect$ by the set of all finite direct sums of
  objects of $\Vect$ and their duals.  Let $G$ be a quasi-coherent $\Orb_X$-module.
  Consider the category $\Vect_G$ whose objects are pairs $(H,\eta)$,
  where $H\in \Vect$ and $\eta\colon H \to G$ is an
  $\Orb_X$-module homomorphism. A morphism
  $h\colon (H,\eta) \to (H',\eta')$ in $\Vect_G$ is an $\Orb_X$-module
  homomorphism $h\colon H \to H'$ such that $\eta = \eta'\circ
  h$. Consider the functor $\mu_G \colon \Vect_G \to \QCoh(X)$ that
  sends $(H,\eta)$ to $H$. It remains to establish the following two
  claims.
  \begin{claim}\label{TC:schappi:colim}
    Every pair of objects in $\Vect_G$ has an upper bound and
    $\colim(\mu_G) \simeq G$.
  \end{claim}
    \begin{proof}
      \renewcommand{\qedsymbol}{$\blacksquare$} If $(H_1,\eta_1)$ and
      $(H_2,\eta_2) \in \Vect_G$, then
      $(H_1\oplus H_2,\eta_1\oplus \eta_2) \in \Vect_G$. This proves
      the existence of upper bounds for every pair. We now prove that
      $\colim(\mu_G) \simeq G$. To do this, consider a quasi-coherent
      $\Orb_X$-module $N$ together with $\Orb_X$-module homomorphisms
      $\nu_{(H,\eta)} \colon H \to N$ for every $(H,\eta) \in \Vect_G$
      such that if $h\colon (H,\eta) \to (H',\eta')$ is a morphism in
      $\Vect_G$, then $\nu_{(H,\eta)} = \nu_{(H',\eta')} \circ h$. It
      suffices to show that there is a uniquely induced morphism
      $\nu \colon G \to N$ such that $\nu_{(H,\eta)} = \nu \circ \eta$
      for every $(H,\eta) \in \Vect_G$. Since $X$ has the $\Vect$-resolution
      property, $G$ admits a presentation:
      \[
        \bigoplus_{j\in J} P_j \xrightarrow{\oplus p_j}
          \bigoplus_{(H,\eta)\in \Vect_G} H \xrightarrow{\oplus_{(H,\eta)\in\Vect_G} \eta} G \to 0,
      \]
      where $P_j$ belongs to $\Vect$ for all $j\in J$. Let
      $j\in J$ and note that the morphism
      $p_j \colon P_j \to \oplus_{(H,\eta)\in \Vect_G} H$ factors as
      \[
        P_j \xrightarrow{\tilde{p}_j} \bigoplus_{(H,\eta)\in I_j} H
        \subseteq \bigoplus_{(H,\eta)\in \Vect_G} H,
      \]
      where $I_j \subseteq \Vect_G$ is finite. Let
      $\tilde{Q}_j = \oplus_{(H,\eta)\in I_j} H$ and let
      $\tilde{q}_j = \oplus_{(H,\eta) \in I_j} \eta \colon \tilde{Q}_j \to G$ be the
      resulting morphism. Then $\tilde{q}_j\circ \tilde{p}_j = 0$ and
      so $\tilde{p}_j\colon (P_j,0) \to (\tilde{Q}_j,\tilde{q}_j)$ is a morphism
      in $\Vect_G$. Hence, $\nu_{(\tilde{Q}_j,\tilde{q}_j)} \circ \tilde{p}_j=\nu_{(P_j,0)}=\nu_{(P_j,0)} \circ 0=0$. But this means $(\oplus_{(H,\eta) \in \Vect_G} \nu_{(H,\eta)}) \circ (\oplus_j {p}_j) = 0$. By the universal property of cokernels, there is a unique
      morphism $\nu \colon G \to N$ such that
      $\nu \circ \eta = \nu_{(H,\eta)}$ for all $(H,\eta)\in \Vect_G$. 
  \end{proof}
  \begin{claim}\label{TC:schappi:filtered}
   If $M$ is a vector bundle
on $Y$ that is a projective object
in $\QCoh(Y)$, then $\Vect_{f_*M}$ is filtered.
  \end{claim}
  \begin{proof}
    \renewcommand{\qedsymbol}{$\blacksquare$}
    Consider a pair of
    morphisms $h_1$, $h_2 \colon (H,\eta) \to (H',\eta')$; we must
    show that these can be coequalized in $\Vect_{f_*M}$. 
    Take the duals of $h_1$ and $h_2$, which results in morphisms
  $h_1^\vee$, $h_2^\vee \colon H'^\vee \to H^\vee$. Let $E$ be
  their equalizer in $\QCoh(X)$. If $(F,\rho) \in \Vect_E$, then
  taking duals of everything results in a commutative diagram:
  \[
    \xymatrix@R-2pc{& H' \ar[dr] \ar@/^1.5pc/[drr]^{\eta'}& & \\H
      \ar[ur]^{h_1} \ar[dr]_{h_2} & & F^\vee  & f_*M \\ &
      H' \ar[ur] \ar@/_1.5pc/[urr]_{\eta'} & & }
  \]
  Thus, we just need to produce $(F,\rho) \in \Vect_E$ that admits a
  compatible morphism $F^\vee \to f_*M$. By taking adjoints in the
  above diagram, we see that it is sufficient to produce a compatible
  morphism $f^*F^\vee \to M$. But $M$ is a vector bundle,
  so it is sufficient to produce a compatible morphism
  $M^\vee \to f^*F$. Dualizing their defining diagrams we obtain:
  \[
    \xymatrix@R-1pc{& & & f^*H'^\vee \ar[dr]^{f^*h_1^\vee} & \\M^\vee  \ar@/^1pc/[urrr] \ar@/_1pc/[drrr]  & f^*F \ar[r]^{f^*\rho} & f^*E \ar[ur] \ar[dr] & & f^*H^\vee.\\ & & & f^*H'^\vee  \ar[ur]_{f^*h_2^\vee} & }
  \]
  But $f$ is flat, so $f^*E$ is also the equalizer of $f^*h_1^\vee$
  and $f^*h_2^\vee$; hence, there is a compatibly induced morphism
  $M^\vee \to f^*E$. By Claim \ref{TC:schappi:colim},
  $E\simeq \colim(\mu_E)$; hence, $f^*E \simeq f^*\colim(\mu_E)$. Now
  $M^\vee$ is a vector bundle and $Y$ is quasi-compact and
  quasi-separated, so the functor
  $\Hom_{\Orb_Y}(M^\vee,-)=\Gamma(Y,M\tensor_{\Orb_Y} - )$ preserves
  filtered colimits of quasi-coherent sheaves
  \cite[Lem.~1.2(iii)]{perfect_complexes_stacks}. Since $M$ is projective in $\QCoh(Y)$, $M^\vee$ is projective in $\QCoh(Y)$.\footnote{This follows from three observations:
    \begin{enumerate*}
    \item a direct summand of a projective is projective;
    \item if $M$ is projective, then $M\tensor_{\Orb_Y} Q$ is
      projective for any vector bundle $Q$; and
    \item $M^\vee$ is a direct summand of
      $M^\vee \tensor_{\Orb_Y} M \tensor_{\Orb_Y} M^\vee$.
    \end{enumerate*}} Hence, $\Hom_{\Orb_Y}(M^\vee,-)$ is also exact and
  so commutes with \emph{all} colimits of quasi-coherent sheaves. Thus we obtain:
  \begin{align*}
    \Hom_{\Orb_Y}(M^\vee,f^*\text{colim}(\mu_E))&\simeq\Hom_{\Orb_Y}(M^\vee,\text{colim}(f^* \mu_E))\\
    &\simeq\text{colim}_{(F,\rho) \in \Vect_E}\Hom_{\Orb_Y}(M^\vee,f^*F).
  \end{align*}
In particular, $M^\vee \to f^*E$ factors
  through some morphism $f^*\rho\colon f^*F \to f^*E$, where
  $(F,\rho) \in \Vect_E$ \cite[Tag \spref{09WR}]{stacks-project}. The
  claim follows. 
\end{proof}
\end{proof}
\begin{Example}
   In Theorem \ref{T:schappi}, it is necessary to include the duals of the generating
set $\mathbf{V}$. Indeed, let $X=\mathbf{P}^1_{x,y}$ over a field $k$ with coordinates $x$ and $y$. Let $Y=\mathrm{D}_+(x)$ and take $f\colon Y\to X$ to be the inclusion. Then $L=\Orb(1)$ is ample and $X$ has the $\mathbf{V} = \{ L^{m}\}_{m\leq 0}$-resolution property (Example \ref{ex:knownres}). But $f_*\Orb_Y\simeq \varinjlim_n L^n$ cannot be written as a filtered colimit of direct sums of objects from $\mathbf{V}$. 
\end{Example}
The converse to Theorem \ref{T:schappi} is an earlier result of Hovey \cite[Prop.~1.4.4]{MR2066503}.
\begin{proposition}\label{P:schappi-converse}
  Let $f \colon Y\to X$ be a faithfully flat affine morphism of quasi-compact and quasi-separated algebraic stacks, where $Y$ is quasi-affine. Let
  $\Vect$ be a set of vector bundles on $X$. If
  $f_*\Orb_{Y}$ is a filtered colimit of duals of finite direct sums of objects of $\Vect$, then $X$ has the $\mathbf{V}$-resolution property.
\end{proposition}
\begin{proof}
  It suffices to prove that $\Vect$ generates $\QCoh(X)$. Write
  $f_*\Orb_Y = \colim_{\lambda\in \Lambda} H_\lambda^\vee$, where
  $H_\lambda=\bigoplus_{V\in J_\lambda} V^{\oplus n_{\lambda}(V)}$, where
  $J_\lambda \subseteq \Vect$ is a finite subset and $n_{\lambda}(V)$ is finite. Let $M \in \QCoh(X)$; then there are natural isomorphisms:
  \begin{align*}
    \Gamma(Y,f^*M) &= \Hom_{\Orb_X}(\Orb_X,f_*f^*M)\\
                   &\simeq \Hom_{\Orb_X}(\Orb_X,f_*\Orb_Y \otimes_{\Orb_X} M) && \mbox{($f$ is affine)}\\
                   &\simeq \Hom_{\Orb_X}(\Orb_X,\colim_\lambda H_\lambda^\vee \otimes_{\Orb_X} M)\\
                   &\simeq \colim_\lambda \Hom_{\Orb_X}(\Orb_X,H_\lambda^\vee \otimes_{\Orb_X} M)\\
                   &\simeq \colim_\lambda \Hom_{\Orb_X}(H_\lambda,M) 
  \end{align*}
  It follows that if $\phi \colon M \to N$ is a homomorphism of
  quasi-coherent $\Orb_X$-modules, then there is a commutative diagram:
  \[
    \xymatrix{\Gamma(Y,f^*M) \ar[r]^-{\simeq} \ar[d]_{\Gamma(Y,f^*\phi)} & \colim_\lambda \Hom_{\Orb_X}(H_\lambda,M) \ar[d]^{\colim_\lambda \Hom_{\Orb_X}(H_\lambda,\phi)}\\
      \Gamma(Y,f^*N) \ar[r]^-{\simeq} & \colim_\lambda
      \Hom_{\Orb_X}(H_\lambda,N).}
  \]
  In particular, if
  $\Hom_{\Orb_X}(V,\phi) \colon \Hom_{\Orb_X}(V,M) \to
  \Hom_{\Orb_X}(V,M)$ is the zero map for all $V \in \mathbf{V}$, then
  $\Gamma(Y,f^*\phi)$ is the zero map. Since $Y$ is quasi-affine,
  $f^*\phi = 0$. But $f$ is faithfully flat, so $\phi = 0$. That is,
  $\mathbf{V}$ generates $\QCoh(X)$.
\end{proof}
Let $S$ be an integral quasi-compact quasi-separated algebraic stack
with generic point $\xi$. We say that a quasi-coherent $\Orb_S$-module
$F$ is \emph{torsion free} if $\ker(F \to (i_{\xi})_*i_\xi^*F) = 0$,
where $i_\xi \colon S_\xi \hookrightarrow S$ is the generic residual gerbe
\cite[Thm.~B.2]{MR2774654}. We record here the following folklore
result.
\begin{corollary}\label{C:res-property-dim1}
  Let $T$ be an integral quasi-compact algebraic stack with affine
  diagonal. If every torsion free quasi-coherent $\Orb_T$-module of
  finite type is a vector bundle, then $T$ has the resolution
  property.
\end{corollary}
\begin{proof}
  Since $T$ is quasi-compact with affine diagonal, there is a smooth
  covering $\tau \colon \spec A \to T$ and $\tau$ is affine. We
  may write $\tau_*\Orb_{\spec A} =\varinjlim_\lambda N_\lambda$,
  where the $N_\lambda$ are finite type quasi-coherent
  $\Orb_T$-submodules of $\tau_*\Orb_{\spec A}$ \cite{rydh-2014}.
  Then the $N_\lambda$ are torsion free, so are all vector bundles. By
  Proposition \ref{P:schappi-converse}, the result follows.
\end{proof}
\begin{remark}\label{R:prufer-cover}
  Corollary \ref{C:res-property-dim1} gives a simple generalization of
  \cite[1.4.5]{MR756316}. Indeed, if $T$ is an integral quasi-compact
  quasi-separated algebraic stack that admits a faithfully flat cover
  by a finite disjoint union of spectra of Dedekind or Pr\"ufer
  domains (i.e., every finitely generated ideal is invertible), then
  every torsion free quasi-coherent $\Orb_T$-module of finite type is
  a vector bundle \cite[Thm.~1]{MR46349}.
\end{remark}
\section{Flags} \label{sec:flag} The goal of this article is to construct
interesting vector bundles on algebraic stacks. The main
problem---even for schemes---is that vector bundles
are glued from local data. The key idea in this paper is to add the
additional structure of a flag to a vector bundle. This ends up being surprisingly useful.

Let $Y$ be an algebraic stack. Let
$V_\bullet$ be a quasi-coherent $\Orb_{Y}$-module $V$ with a finite filtration:
\[
  0 = V_{0} \subseteq V_1 \subseteq \cdots \subseteq V_{n-1} \subseteq V_n = V.
\]
We call $n$ the \emph{length} of the filtration. Recall that
$V_\bullet$ is a \emph{flag} if the graded pieces
$\gr_i(V_\bullet)=V_{i}/V_{i-1}$ are vector bundles for all $i=1$,
$\dots$, $n$. Note that this condition implies that the $V_i$ are
also vector bundles. A flag is \emph{complete} if the graded pieces
are line bundles. We say that $V_\bullet$ is \emph{trivially graded}
if $\gr_i(V_{\bullet})$ is a trivial vector bundle for all $i$. Note that a vector bundle which admits the structure of a trivially graded flag is sometimes referred to as a unipotent vector bundle in the literature (see \cite[Sec. 1]{MR318151} and \cite[Def. 4.5]{MR498572}). We now
have the following key definition of the paper.
\begin{definition}
  Let $f\colon X \to S$ be a morphism of algebraic stacks. Let
  $V_\bullet$ be a flag on $X$. We say that $V_\bullet$ is
  \emph{$f$-graded} (or \emph{$S$-graded}) if for each $i$ there is a
  vector bundle $E_i$ on $S$ and an isomorphism
  $\phi_i \colon \gr_i(V_\bullet)\simeq f^*E_i$. If the vector bundle $V$ is
   $f$-(R-)faithful, then we say the same of the flag $V_\bullet$.
\end{definition}
Note that trivially graded flags are graded by every morphism.
Remarkably, we can establish that graded flags often descend (Theorem
\ref{T:descent-flags}). We are not aware of any related result in the
literature. We begin, however, with the universal example.

\begin{Example}\label{E:unipotent-triv-bundle}
  Let $S$ be an algebraic stack. Fix an ordered sequence of $n$ line bundles
  $\mathscr{L}=(L_1,\dots,L_n)$ on $S$. Define a category fibered in
  groupoids $\FLAG_{\mathscr{L}} \to \SCH{S}$ as follows: its
  objects over $f \colon X \to S$ are pairs $(V_\bullet, \{ \phi_i\}_{i=1}^n)$,
  where
  \begin{enumerate}
  \item $V_\bullet$ is a flag of length $n$ and
  \item $\phi_i\colon \gr_i(V_\bullet) \simeq f^*L_i$ are isomorphisms
    for $i=1$, $\dots$, $n$.
  \end{enumerate}
  That is, an $X$-point of $\FLAG_{\mathscr{L}}$ is an $S$-graded flag
  on $X$ whose graded pieces are isomorphic to the $f^*L_i$. A morphism
  in $\FLAG_{\mathscr{L}}$ is an isomorphism of vector bundles that is
  compatible with the filtrations and isomorphisms to the $L_i$.
  Note that there is a distinguished object in
  $\FLAG_{\mathscr{L}}$ defined over $S$:
  \[
    \ell=(L_1 \subseteq \cdots \subseteq \bigoplus_{j=1}^{n-1} L_j \subseteq \bigoplus_{j=1}^n L_j,\{\text{id}_{L_i}\}_{i=1}^n).
  \] 
  Every other object in $\FLAG_{\mathscr{L}}$ becomes isomorphic to
  $\ell$ after passing to a smooth covering. In particular,
  $\FLAG_{\mathscr{L}}$ is a gerbe and the section defined by $\ell$
  induces an equivalence
  $\FLAG_{\mathscr{L}} \simeq \classB\U(\mathscr{L})$, where
  $\U(\mathscr{L})=\Aut(\ell)$. Clearly,
  $\U(\mathscr{L}) \subseteq \GL(\bigoplus_{j=1}^n L_j)$. Moreover, the following diagram $2$-commutes:
  \[
    \xymatrix{X \ar[r]^{\Fr(V)} \ar[d]_-{(V_\bullet, \{\phi_i\})} & \classB \GL_{n,S}\\
    \classB \U(\mathscr{L}) \ar[r] & \classB \GL(\bigoplus_{j=1}^n L_j) \ar[u]^-{\rotatebox{90}{$\sim$}}_{\Fr(\bigoplus_j L_j)}.}
  \]
  If
  $\mathscr{L}=(\mathcal{O}_S,...,\mathcal{O}_S)$, then there is a
  natural identification
  \[\Aut(\ell) \simeq \U_{n,S}.\]
  Thus for any choice of $\mathscr{L}$, $\Aut(\ell)$ is locally
  isomorphic to $\U_{n,S}$. In particular, $\U(\mathscr{L}) \to S$ is a flat group of finite presentation and it follows from Proposition
  \ref{P:hom} that
  $\GL(\bigoplus_{j=1}^n L_j)/\U(\mathscr{L}) \to S$ is quasi-affine,
  and $\U(\mathscr{L}) \to S$ is R-embeddable. Hence,
  $\classB\U(\mathscr{L})$ is an algebraic stack, the morphism
  $\classB\U(\mathscr{L}) \to S$ has the resolution property and the
  universal flag $F_\bullet(\mathscr{L})$ on $\classB \U(\mathscr{L})$
  is a $(\classB \U(\mathscr{L}) \to S)$-tensor generator. In
  particular, if $S$ has the resolution property, then so does
  $\classB \U(\mathscr{L})$.

  If $\B_{n,\Z}$ denotes the Borel subgroup of upper triangular
  matrices in $\GL_{n,\Z}$ and $\B_{n,\Z} \to \mathbf{G}_m^n$ is the
  induced quotient, then the following diagram is $2$-cartesian:
  \begin{equation}
    \vcenter{\xymatrix{\classB \U(\mathscr{L}) \ar[r] \ar[d] & \classB \B_{n,\Z} \ar[d]\\ S \ar[r]_{\mathscr{L}} & \classB \mathbf{G}_m^n.}}\label{eq:ubn}
  \end{equation}
\end{Example}

Let $f\colon X \to S$ be a
morphism of algebraic stacks and let $\textbf{Vect}(X)$ denote the set of isomorphism classes of vector bundles on $X$. We let
\[\Flag_{\mathcal{O}} \subseteq\Flag_f \subseteq\Flag\subseteq\textbf{Vect}(X)\]
denote the sets of isomorphism classes of vector
bundles that admit trivially graded complete flags, complete $f$-graded flags, and complete flags respectively.

\begin{Example} \label{ex:polynomial} Standard arguments show that  $\Flag$,
$\Flag_f$, and $\Flag_{\mathcal{O}}$ are stable under finite direct sums, the taking of duals, finite tensor products, and extensions. 
\end{Example} 

\begin{Example}\label{E:flag-resolution}
  Let $X$ be an algebraic stack. If
  \begin{enumerate}
  \item $X$ is affine; or 
  \item $X$ is quasi-affine; or
  \item $X$ is quasi-projective over affine; or
  \item $X$ admits an ample family of line bundles; or 
  \item $\QCoh(X)$ is generated by a set of line bundles (e.g.,
    $X=\classB \mathbf{G}_{m, \mathbf{Z}}^n$); or
  \item $X=\classB \U_{n,k}$, where $k$ is a field;
  \end{enumerate} 
  then $X$ has the $\Flag$-resolution property. These assertions are
  trivial. If $X$ has the $\Flag$-resolution property, then
  Sch\"appi's Theorem (Theorem \ref{T:schappi}) implies that if
  $p \colon \Spec A \to X$ is flat, then
  $p_*\Orb_{\Spec A} \simeq \varinjlim_{\lambda} F_\lambda$, where the
  $F_\lambda$ are complete flags on $X$. Note that
  $\classB \GL_{n,\Z}$ does not have the $\Flag$-resolution property
  if $n>1$.
\end{Example}
The following result yields a useful sufficient condition
for the existence of a flag structure on a vector bundle. 
\begin{lemma} \label{L:pruferflag} Let $T$ be an integral
  quasi-compact quasi-separated algebraic stack with generic point
  $\eta \colon \spec k \to T$, where $k$ is a field. Let $\gamma \colon G \to T$ be
  a flat group of finite presentation and let
  $f \colon \classB G \to T$ be the induced morphism. Let $V$ be a
  vector bundle on $\classB G$.
  \begin{enumerate}
  \item \label{LI:pruferflag:adjunction} If the adjunction
    $z \colon f^*f_*V \to V$ is an isomorphism after restriction
    along $\eta$, then it is an isomorphism.
  \item \label{LI:pruferflag:flag} If every torsion free
    quasi-coherent $\Orb_T$-module of finite type is a vector bundle
    (Remark \ref{R:prufer-cover}) and $G_\eta$ is
    unipotent, then $V$ admits an $f$-graded flag.
  \end{enumerate}
\end{lemma}
\begin{proof}  
  By \cite[Thm.~B.2]{MR2774654}, $\eta$ factors as
  $\spec k \xrightarrow{\eta'} T_\eta \xrightarrow{i_\eta} T$, where
  $T_\eta$ is the residual gerbe and $\eta'$ is faithfully flat. For
  \eqref{LI:pruferflag:adjunction}: since $z_\eta$ is an isomorphism,
  $z_{T_\eta}$ is an isomorphism. Let $p\colon T \to \classB G$ be the
  usual section to $f$; by descent, $V$ is described by a group
  homomorphism $v \colon G \to \GL(p^*V)$. Note that $z$ is an isomorphism if
  and only if $\ker(v) = G$ as $z$ corresponds to the inclusion of
  $G$-invariants of $V$ into $V$. Now $\ker(v) \subseteq G$ is a
  closed immersion as $\GL(p^*V) \to T$ is separated and contains
  $\gamma^{-1}(T_\eta)$, which is dense since $\gamma$ is open.  Hence
  $\ker(v) = G$.

  We prove \eqref{LI:pruferflag:flag} by induction on $\rank (V)$, the
  case $\rank(V) = 0$ being trivial; so we now assume that
  $\rank(V) > 0$. If $z_{T_\eta}$ is an isomorphism, then
  \eqref{LI:pruferflag:adjunction} says that $z_V$ is an isomorphism
  and so we would be done. Otherwise, if $z_{T_\eta}$ is not an
  isomorphism, then $(f^*f_*V)_{T_\eta}$ is non-zero as after restriction along $\eta$ it is the
  inclusion of the invariants of a non-trivial representation of the
  unipotent group $G_{\eta}$ \cite[Thm.~XVII.3.5]{SGA3-III-NEW}. In
  particular, there is a non-trivial quotient $V_{T_\eta} \to W$ on
  $T_\eta$. Now define $\overline{W}$ to be the image of the
  composition $V \to (i_{T_\eta})_*V_{T_\eta} \to
  (i_{T_\eta})_*W$. Since the map $V \to (i_{T_\eta})_*W$ is non-zero,
  $\overline{W}$ is a non-zero, torsion free, quasi-coherent
  $\Orb_{\classB G}$-module of finite type, and so is a vector
  bundle. Also, $\overline{W}_{T_\eta} = W$ and so
  $\rank(\overline{W}) = \rank(W) < \rank(V)$.  We now have an
  exact sequence of vector bundles on $\classB G$:
  \[
    \xymatrix{0 \ar[r] & K \ar[r] & V \ar[r] & \overline{W} \ar[r] & 0.}
  \]
  The ranks of $K$ and $\overline{W}$ are less than that of $V$ and so
  by induction, they admit $f$-graded flags and consequently so does
  $V$.
\end{proof}
Using Lemma \ref{L:pruferflag}, we can now prove the following
Proposition.
\begin{proposition} \label{P:ulflag} Let $\mathscr{L}$ be a sequence
  of $n$ line bundles on an algebraic stack $S$.
  \begin{enumerate}
  \item \label{PI:ulflag:flagf} The morphism $f\colon \classB \U(\mathscr{L})\to S$ has the
    $\Flag_f$-resolution property.
\item \label{PI:ulflag:flagO}If $S$ is affine, then $f$ has the $\Flag_{\mathcal{O}}$-resolution property.
\end{enumerate}
\end{proposition}
\begin{proof}
  It suffices to show \eqref{PI:ulflag:flagf} universally
  \cite[Prop. 2.8(iv)]{2013arXiv1306.5418G}. Hence, we may suppose
  $f\colon \classB\U(\mathscr{L}) \to \classB\G_{m,\Z}^n$, where
  $\mathscr{L}$ is the universal sequence of $n$ line bundles. Since
  $\classB \B_{n,\Z} \simeq B\U(\mathscr{L})$ has the resolution
  property (Example \ref{E:unipotent-triv-bundle}), it suffices to
  show that every vector bundle $V$ is a complete $f$-graded
  flag. Since $\spec \Z \to \classB \G_{m,\Z}^n$ is a smooth covering,
  Lemma \ref{L:pruferflag} and Remark \ref{R:prufer-cover} imply that
  $V$ admits an $f$-graded flag. But every vector bundle on
  $\classB\G_{m,\Z}^n$ is a direct sum of line bundles, so
  \eqref{PI:ulflag:flagf} follows from Example \ref{ex:polynomial}.

  For \eqref{PI:ulflag:flagO}, it suffices to show that the vector bundle $V$ underlying every $f$-graded complete flag
  $V_{\bullet}$ is a direct summand of a vector bundle $W$ which admits a trivially graded complete flag structure. We proceed
  by induction on $n$, the length of $V_\bullet$. If $n=0$, then the
  result is trivial. If $n>0$, then we may write $V$ as an
  extension
  \[
    \xymatrix{0 \ar[r] & V_{n-1} \ar[r] & V \ar[r] & f^*L
      \ar[r] & 0.}
  \] 
  for some line bundle $L$ on $S$. Since $S$ is affine, there is a
  split surjection
  $\gamma \colon \Orb_S^{\oplus m} \twoheadrightarrow L$. Then $f^*\gamma \colon \Orb^{\oplus m}_{\classB \U(\mathscr{L})} \twoheadrightarrow f^*L$ is a split surjection. Now pull the extension back along $f^*\gamma$, so we may replace $f^*L$ with a
  trivial bundle of rank $m$. By induction, there is a split
  surjection $\delta \colon U \twoheadrightarrow V_{n-1}$, where $U$
  admits a trivially graded complete flag. Pushing forward along a section of $\delta$,
  $V$ is a direct summand of a vector bundle that is an
  extension of trivially graded complete flags. Now apply Example
  \ref{ex:polynomial}. \end{proof}

The following concept will be useful.
A \emph{refinement} of a flag $V_\bullet$ of length
$n$ is another flag 
$V_0' \subseteq \cdots \subseteq V_m' = V$ of length $m\geq n$
together with a strictly order preserving function
$\rho \colon \{0<1<\cdots<n\} \to \{0< 1 < \cdots < m\}$ such that
$V_i = V'_{\rho(i)}$ for all $i$. 

\begin{Example}\label{E:complete-refinement}
  Let $Y$ be an algebraic stack. Let $V_\bullet$ be a flag on
  $Y$. Suppose that $\gr_i(V_\bullet)$ admits the structure of a complete flag for each
  $i\geq 0$; then $V_\bullet$ admits a complete refinement. To see
  this, use Example \ref{ex:polynomial}. \end{Example}

\begin{Theorem}\label{T:descent-flags}
  Consider a $2$-cartesian diagram of algebraic stacks:
  \[
    \xymatrix{X' \ar[r]^q \ar[d]_{f'} & X \ar[d]^{f}\\ S' \ar[r]^p & S.}
  \]
  Assume that
  \begin{enumerate}
  \item $S$ is quasi-compact with affine diagonal and has the resolution property;
  \item $f$ is quasi-compact and quasi-separated;  
  \item $p$ is faithfully flat and $S'$ is affine.
  \end{enumerate}
  If $X'$ admits a $q$-graded flag $V_\bullet$ of length $n$, then $X$
  admits a flag $W_\bullet$ of length $n$ and a degree preserving
  split surjection $q^*W_\bullet \twoheadrightarrow
  V_\bullet$. Moreover,
  \begin{enumerate}[label=(\alph*)]
  \item \label{TI:descent-flags:faithful} if
    $V_\bullet$ is $f'$-(R-)faithful, then $W_\bullet$ is
    $f$-(R-)faithful. 
  \item \label{TI:descent-flags:explicit} If
    $p_*\Orb_{S'}$ is a filtered colimit of vector bundles $F_\lambda$ on $X$ (see Theorem \ref{T:schappi}) and
    $\gr_i(V_\bullet) \simeq q^*E_i$, then there exists $\lambda_i$
    with
    $\gr_i(W_\bullet) \simeq f^*F_{\lambda_i}\tensor_{\Orb_X} E_i$.
  \item \label{TI:descent-flags:trivial} If $V_\bullet$ is
    $fq$-graded, then $W_\bullet$ is $f$-graded.
  \item \label{TI:descent-flags:complete} Assume that $S$ has the
    $\Flag$-resolution property.
    \begin{enumerate}[label=(\roman*)]
    \item If $V_\bullet$ is complete, then
      $W_\bullet$ admits a complete refinement.
    \item If $V_\bullet$ is $fq$-graded, then $W_\bullet$ admits a
      complete $f$-graded refinement.
  \end{enumerate}
\end{enumerate}
\end{Theorem}
\begin{proof}
  Claim \ref{TI:descent-flags:faithful} follows from Proposition
  \ref{prop:framebundle}\eqref{prop:framebundle:sum}, the main claim, and \cite[Prop. 2.8(iii)]{2013arXiv1306.5418G}. Claim
  \ref{TI:descent-flags:trivial} follows from
  \ref{TI:descent-flags:explicit}. Claim
  \ref{TI:descent-flags:complete} follows from
  \ref{TI:descent-flags:explicit} and Examples
  \ref{E:complete-refinement} and \ref{E:flag-resolution}. Claim
  \ref{TI:descent-flags:explicit} will follow from the construction of
  $W_\bullet$, which we prove by induction on $n\geq 0$. The base case is trivial so we assume $n \geq 1$. Choose a
  vector bundle $E$ on $X$ such that $V_1=q^*E$, then $(V/V_1)_\bullet$ is a
  $q$-graded flag of length $n-1$. By induction, there is a flag
  $W_\bullet$ of length $n-1$ on $X$ and a split surjection
  $q^*W_\bullet \to (V/V_1)_\bullet$ as in the statement of the theorem. We now pull back the defining
  short exact sequence
  \[
    \xymatrix{ 0 \ar[r] & q^*E \ar[r] & V \ar[r] & V/q^*E \ar[r] & 0}
  \]
  along the surjection $q^*W \twoheadrightarrow V/q^*E$. This results in a
  commutative diagram with exact rows, whose vertical arrows are
  easily checked to be split surjective:
  \[
    \xymatrix{ 0 \ar[r] & q^*E \ar@{=}[d] \ar[r] & V' \ar[r]
      \ar@{->>}[d]& q^*W \ar@{->>}[d] \ar[r] & 0  \\ 0 \ar[r] &
      q^*E \ar[r] & V \ar[r] & V/q^*E \ar[r] & 0.}
  \]
  Thus, we may replace $V$ by $V'$ and assume that
  $(V/V_1)_\bullet=q^*W_\bullet$. 

  Since
  $S'$ is affine and $S$ has affine diagonal, it follows that $p$ and
  $q$ are affine and faithfully flat morphisms. It follows that
  we may push forward the top row in the above diagram to obtain an
  exact sequence of quasi-coherent sheaves on $X$:
  \[
    \xymatrix{0 \ar[r] & q_*q^*E \ar[r] & q_*V'
      \ar[r] & q_*q^*W \ar[r] & 0.}
  \]
  Pulling this sequence back along the injective adjunction
  $W \hookrightarrow q_*q^*W$, we obtain a diagram with exact rows:
  \[
    \xymatrix{0 \ar[r] & q_*q^*E \ar@{=}[d] \ar[r] &
      W' \ar[r] \ar@{^(->}[d] & W \ar@{^(->}[d] \ar[r] & 0 \\0 \ar[r] &
      q_*q^*E \ar[r] & q_*V' \ar[r] &
      q_*q^*W \ar[r] & 0.}
  \]
  Since $E$ is a vector bundle and $q$ is affine, the projection formula and flat base
  change produce natural isomorphisms:
  \[
    q_*q^*E \simeq q_*\Orb_{X'}
    \tensor_{\Orb_{X}} E \simeq (f^*p_*\Orb_{S'}) \tensor_{\Orb_{X}} E.
  \]
  Assume we are given that
  $p_*\Orb_{S'} \simeq \varinjlim_\lambda F_\lambda$, where the
  $F_\lambda$ are vector bundles, or we produce such a description by
  Sch\"appi's Theorem (Theorem \ref{T:schappi}). Since the adjunction
  $p^*p_*\Orb_{S'} \twoheadrightarrow \Orb_{S'}$ is split surjective,
  taking $\lambda$ sufficiently large, we can assume that the induced
  morphisms $p^*F_\lambda \to \Orb_{S'}$ are also split
  surjective. But we also have natural isomorphisms:
  \begin{align*}
    \Ext^1_{\Orb_{X}}(W,q_*q^*E) &\simeq \H^1(X,W^\vee \tensor_{\Orb_{X}} q_*q^*E)\\
                                 &\simeq \H^1(X, \varinjlim_\lambda (f^*F_\lambda \tensor_{\Orb_{X}} W^\vee \tensor_{\Orb_{X}} E))\\
                                 &\simeq \varinjlim_\lambda \H^1(X,f^*F_\lambda \tensor_{\Orb_{X}} W^\vee \tensor_{\Orb_{X}} E) &&\mbox{\cite[Lem.~1.2(iii)]{perfect_complexes_stacks}}\\
                                 &\simeq \varinjlim_\lambda \Ext^1_{\Orb_{X}}(W,f^*F_\lambda \tensor_{\Orb_{X}} E). 
  \end{align*}
  Thus, the
  extension $W'$ is the push forward of an extension:
  \[
    \xymatrix{0 \ar[r] & f^*F_\lambda \tensor_{\Orb_{X}} E\ar[r] & W'_\lambda \ar[r] & W \ar[r] & 0}
  \]
  along the morphism
  $f^*F_\lambda \tensor_{\Orb_{X}} E \to f^*(p_*\Orb_Y)
  \tensor_{\Orb_{X}} E \simeq q_*q^*E$. Pulling this
  exact sequence back along $q$, we obtain an exact sequence:
  \[
    \xymatrix{0 \ar[r] & q^*f^*F_\lambda \tensor_{\Orb_{X'}} q^*E\ar[r] & q^*W'_\lambda \ar[r] & q^*W \ar[r] & 0}
  \]
  But the push forward of this extension along the split surjection
  $q^*f^*F_\lambda \tensor_{\Orb_{X'}} q^*E
  \twoheadrightarrow q^*E$ is clearly isomorphic to $V'$ and the
  resulting morphism $q^*W'_\lambda \twoheadrightarrow V'$ is split
  surjective. The result follows.
\end{proof}
\section{Unipotent morphisms and groups}\label{sec:unipotent-morphisms-groups}

\subsection{Unipotent morphisms}
Motivated by the results of the previous section, we make the following definition.
\begin{Definition} \label{def:unipotent} Let $f \colon X \to S$ be a
  morphism of algebraic stacks. Then $f$ is:
  \begin{itemize}
  \item \emph{(R-)unipotent} if $X$ admits a
    complete $f$-graded $f$-\mbox{(R-)}faithful flag;
  \item \emph{locally (R-)unipotent} if there is an fpqc covering
    $S' \to S$ \cite[Tag \spref{022B}]{stacks-project} such that
    $f\times_S S'$ is (R-)unipotent; and
  \item \emph{geometrically (R-)unipotent} if for every algebraically closed field $k$ and
    morphism $\Spec k \to S$, the induced morphism
    $X\times_S \Spec k \to \Spec k$ is (R-)unipotent.
  \end{itemize}
  Locally R-unipotent and unipotent morphisms have affine and
  quasi-affine diagonals respectively.  We have the following
  sequence of implications:
  
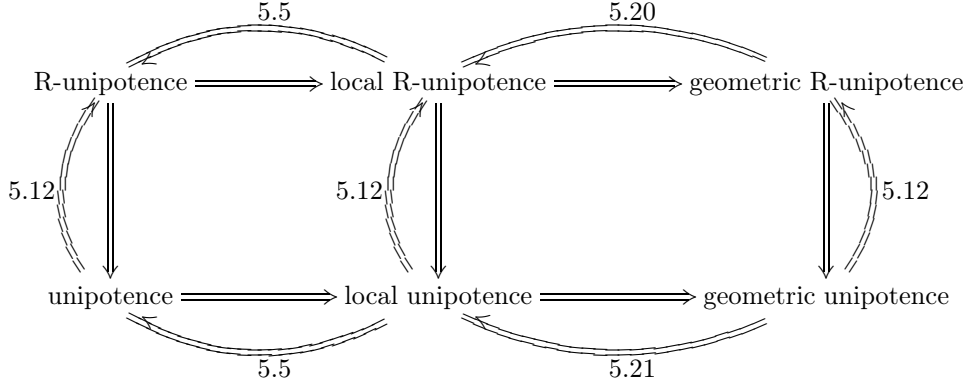
\[
    \xymatrix@R+0.5pc@C+2pc{\mbox{R-unipotence}\ar@{=>}[r]
      \ar@{=>}[dd] & \mbox{local R-unipotence}
      \ar@/_1.75pc/@{==>}[l]_{\mbox{\ref{T:flat-local-unipotent-morphism}}}
      \ar@{=>}[r] \ar@{=>}[dd] & \mbox{geometric R-unipotence} \ar@{=>}[dd]  \ar@/_1.75pc/@{==>}[l]_{\mbox{\ref{E:char0-affine}}}\\
      & & \\ \mbox{unipotence}
      \ar@/^1.5pc/@{==>}[uu]^{\mbox{\ref{R:unipotent-group-thesame}}}
      \ar@{=>}[r]&
      \ar@/^1.75pc/@{==>}[l]^{\mbox{\ref{T:flat-local-unipotent-morphism}}}
      \mbox{local unipotence}
      \ar@/^1.5pc/@{==>}[uu]^{\mbox{\ref{R:unipotent-group-thesame}}}
      \ar@{=>}[r] &
      \ar@/^1.75pc/@{==>}[l]^{\mbox{\ref{E:examples-r-unipotence}}}
      \ar@/_1.5pc/@{==>}[uu]_{\mbox{\ref{R:unipotent-group-thesame}}}\mbox{geometric
        unipotence}}
  \]
  The solid
  arrows follow from the definitions. The dashed arrows are partial
  converses. The dashed vertical implications are valid only for certain groups over normal bases. The dashed
  horizontal implications are valid only in special cases.
\end{Definition}
Quasi-compact and quasi-separated representable morphisms are always unipotent. There are two
prototypical examples of R-unipotent morphisms:
\begin{enumerate}
\item $\classB \U(\mathscr{L}) \to S$ (Example
  \ref{E:unipotent-triv-bundle}), and
\item quasi-affine morphisms $X\to S$.
\end{enumerate}
We will return to these examples often. 
We begin this section with the
following two trivial lemmas.
\begin{lemma}\label{L:basechange-unipotent}
  Let $f \colon X \to S$ be a (R-)unipotent morphism of algebraic
  stacks. If $S' \to S$ is a morphism of algebraic stacks, then
  $f\times_S S'$ is (R-)unipotent. Moreover, the same holds for the
  local and geometric versions.
\end{lemma}
\begin{lemma}\label{L:restriction-flag}
  Let $X \xrightarrow{h} Y \xrightarrow{g} S$ be 
  morphisms of algebraic stacks.
  \begin{enumerate}
  \item\label{LI:restriction-flag:uni} If $g$ is unipotent and $h$ is quasi-compact, quasi-separated and representable; then
    $g\circ h$ is unipotent.
  \item\label{LI:restriction-flag:Runi} If $g$ is R-unipotent and $h$
    is quasi-affine, then $g\circ h$ is R-unipotent.
  \end{enumerate}
  Moreover, the analogous properties hold for the local and geometric
  versions.
\end{lemma}
In general, unipotent morphisms are not stable under composition
(Example \ref{E:unstable-unipotent-composition}). We are optimistic
that locally unipotent morphisms have better stability properties under composition, but
this appears to be surprisingly subtle---even when $S$ is the spectrum
of a field, $Y$ is quasi-affine and $X$ is a gerbe over $Y$.

We now have the
following characterization of (R-)unipotent morphisms, which provides
a unipotent enrichment of the Totaro--Gross Theorem (Theorem
\ref{thm:totgross}).
\begin{Theorem}\label{T:char-Runi}
  Let $f\colon X \to S$ be a morphism of algebraic stacks with $S$
  quasi-compact. Then the following are equivalent.
  \begin{enumerate}
  \item\label{T:char-Runi:map} The morphism $f$ is R-unipotent.
  \item\label{T:char-Runi:factor} For some integer $n \geq 0$, there
    is an ordered sequence of $n$ line bundles $\mathscr{L}$ on $S$
    and factorization of $f$ as
    $X \xrightarrow{h} \classB \U(\mathscr{L}) \to S$, where $h$ is
    quasi-affine.
  \item\label{T:char-Runi:flag} The morphism $f$ has the $\Flag_f$-resolution
    property and the relative inertia stack $I_f \to X$ has affine
    fibers.
  \end{enumerate} 
  Moreover, $f$ is unipotent if and only there is a quasi-compact, quasi-separated and representable morphism $X \xrightarrow{h} \classB \U(\mathscr{L})$
  as in \eqref{T:char-Runi:factor}.
\end{Theorem}
\begin{proof}
  The statement regarding unipotent morphisms will follow from our
  arguments from the equivalence
  \eqref{T:char-Runi:map}$\Leftrightarrow$\eqref{T:char-Runi:factor}.

  We have \eqref{T:char-Runi:factor}$\Rightarrow$\eqref{T:char-Runi:map} by Lemma \ref{L:restriction-flag} and the
  R-unipotence of $\classB\U(\mathscr{L}) \to S$. As for
  \eqref{T:char-Runi:map}$\Rightarrow$\eqref{T:char-Runi:factor}: by
  definition, there is a complete $f$-graded flag $V_\bullet$ of
  length $n$ on $X$ that is $f$-(R-)faithful. That is, there is an
  ordered sequence of line bundles $\mathscr{L} = (L_1,\dots,L_n)$ on
  $S$ together with isomorphisms
  $\phi_i \colon \gr_i(V_\bullet) \simeq f^*L_i$ for all $i$. By
  Example \ref{E:unipotent-triv-bundle}, we obtain a morphism
  $X \xrightarrow{h} \classB \U(\mathscr{L})$. Post-composing this
  with the quasi-affine morphism
  $\classB \U(\mathscr{L}) \to \classB \GL(\bigoplus_{i=1}L_i) \simeq
  \classB \GL_{n,S}$ yields a quasi-affine morphism by Theorem \ref{thm:totgross}\eqref{thm:totgross:frame}, and the result follows from \cite[Tag 054G]{stacks-project}.
  
  For \eqref{T:char-Runi:factor}$\Rightarrow$\eqref{T:char-Runi:flag}: Proposition \ref{P:ulflag} and
  \cite[Prop.\ 2.8(v)]{2013arXiv1306.5418G} implies that $f$ has the
  $\Flag_f$-resolution property. Finally, for
  \eqref{T:char-Runi:flag}$\Rightarrow$\eqref{T:char-Runi:factor}: the
  Gross--Totaro Theorem (Theorem \ref{thm:totgross}) produces a
  $V_{\bullet} \in \Flag_f$ which is $f$-R-faithful.
\end{proof}
A remarkable consequence of Theorem \ref{T:descent-flags} is
that local (R-)unipotence implies (R-)unipotence on bases with the
$\Flag$-resolution property.
\begin{Theorem}\label{T:flat-local-unipotent-morphism}
  Let $f\colon X \to S$ be a locally unipotent (resp., locally R-unipotent)
  morphism of algebraic stacks. Let $S$ be quasi-compact with affine
  diagonal.
  \begin{enumerate}
  \item If $S$ has the resolution property, then $X$ is a global quotient
    (resp., has the resolution property).
  \item If $S$ has the $\Flag$-resolution property, then $f$ is unipotent (resp., R-unipotent). 
\end{enumerate}
\end{Theorem}
\begin{proof}
  By assumption, there is a faithfully flat cover $p \colon S' \to S$
  such that $f' \colon X\times_S S' \to S'$ admits a complete $f'$-graded
  $f'$-faithful (resp., $f'$-R-faithful) flag $V_\bullet$.  Passing to a smooth cover of $S'$,
  we may assume that $S'$ is affine and that the flag is trivially graded, and so $V_\bullet$ is
  $pf'$-graded. Now apply Theorem \ref{T:descent-flags} to:
  \[
    \xymatrix{X\times_S S' \ar[r]^-{q} \ar[d]_{f'} & X \ar[d]^f
        \\ S' \ar[r]_p & S.} 
  \]
  Then we obtain a flag $W_\bullet$ on $X$ and a split surjection
  $q^*W_\bullet \to V_\bullet$. Part \ref{TI:descent-flags:trivial}
  implies that $W_\bullet$ is $f$-graded. Part
  \ref{TI:descent-flags:faithful} says that $W_\bullet$ is
  $f$-faithful (resp., $f$-R-faithful), so $X$ is a global quotient (resp., has the resolution
  property). Part \ref{TI:descent-flags:complete} says that if $S$ has the
  $\Flag$-resolution property, then $W_\bullet$ admits a complete
  refinement, which is also $f$-graded. Hence, $f$ admits a complete
  $f$-graded $f$-(R-)faithful flag. That is, $f$ is unipotent (resp., R-unipotent).
\end{proof}
\begin{Example}\label{E:unstable-unipotent-composition}
  Let $S$ be a proper, normal surface over $\C$
  with no non-trivial line bundles
  \cite{schroer-non-proj-surfaces}. Let $Y \subseteq S$ be the inclusion of the regular
  locus, which is an open subscheme whose complement $S\setminus Y$
  is a finite set of closed points. Then $Y$ admits a non-trivial very ample line
  bundle $L$. Let $\ms U = \classB \U(\Orb,L^\vee) \cong BL$, then
  \begin{enumerate}
  \item \label{EI:unstable-unipotent-composition:ind}
    $\pi \colon \ms U \to Y$ and $j \colon Y \to S$ are
    R-unipotent;
  \item \label{EI:unstable-unipotent-composition:comp} the composition $p=j \circ \pi$ is locally R-unipotent but not
    unipotent; and
  \item \label{EI:unstable-unipotent-composition:flag-S} $S$ has the
    resolution property but not the $\Flag$-resolution property.
  \end{enumerate}
  Indeed, \eqref{EI:unstable-unipotent-composition:ind} is clear from
  Theorem \ref{T:char-Runi}. Also,
  \eqref{EI:unstable-unipotent-composition:flag-S} follows from
  \eqref{EI:unstable-unipotent-composition:comp} and Theorem
  \ref{T:flat-local-unipotent-morphism}.

  For \eqref{EI:unstable-unipotent-composition:comp}: to see that $p$
  is locally R-unipotent, it suffices to restrict to affine opens
  $\spec A \subseteq S$; then $Y_A = Y\cap \spec A$ is quasi-affine
  and so $L_{Y_A}^{\vee}$ is globally generated. Hence, there is a
  quasi-affine morphism $\ms U \to B\mathbf{G}_{a,Y_A}^n$ for some
  $n$; after composition with the open immersion
  $B\mathbf{G}_{a,Y_A}^n \subseteq B\mathbf{G}_{a,\spec A}^n$, we see
  that $p$ is locally R-unipotent by Lemma \ref{L:restriction-flag}\eqref{LI:restriction-flag:Runi}.

  If $p=j\circ \pi$ is unipotent, since $S$ has no non-trivial line
  bundles, $\ms U$ must admit a faithful trivially graded flag. Hence,
  it suffices to show that every trivially graded flag $V_\bullet$ on
  $\ms U$ is of the form $\pi^*W_{\bullet}$ for some trivially graded
  flag $W_\bullet$ on $Y$. A simple calculation shows that
  $\RDERF^1 \pi_*\Orb_{\ms U} \simeq L^\vee$. Note that $Y$ is not quasi-affine because $\H^0(Y,\mathcal{O}_Y)=\C$ \cite[Tag \spref{01P9}]{stacks-project}, and therefore $L$ is non-trivial  \cite[Tag \spref{01QE}]{stacks-project}. In turn, this implies $\H^0(Y,L^{\vee})=0$ because by composing with a non-zero section of $L$, a non-zero section $L^{\vee}$ would yield a trivialization of $L^{\vee}$. Therefore, since
  $L^{\vee}$ has no non-zero global sections, if $W_{\bullet}$ is a
  trivially graded flag $W_\bullet$ on $Y$, then
  $\Hom_{\Orb_{Y}}(W,L^\vee)=0$. Now let $V_\bullet$ be a trivially
  graded flag on $\ms U$. By induction on the length of $V_\bullet$,
  we may assume that $V_\bullet$ is an extension:
  \[
  e\colon  0 \to \pi^*\Orb_Y \to V_\bullet \to \pi^*W_\bullet \to 0,
  \]
  where $W_\bullet$ is a trivially graded flag on $Y$. There is also an exact sequence:
  \[
    0 \to \Ext^1_{\Orb_{Y}}(W,\Orb_Y) \to
    \Ext^1_{\Orb_{\ms U}}(\pi^*W,\pi^*\Orb_Y)
    \to \Hom_{\Orb_Y}(W,\RDERF^1\pi_*\Orb_{\ms U}).
  \]
  Since $\Hom_{\Orb_{Y}}(W,L^\vee)=0$, the first map is
  bijective and so our extension $e$ is pulled back from $Y$. 
\end{Example}
A useful application of Theorem
\ref{T:flat-local-unipotent-morphism} is the following.
\begin{corollary}\label{C:flat-local-implies-smooth-local-morphism}
  Let $f\colon X \to S$ be a locally (R-)unipotent
  morphism of algebraic stacks. If $S$ is quasi-compact, then there is a smooth surjection
  $S' \to S$ such that $f\times_S S'$ is (R-)unipotent.
\end{corollary}
\begin{proof}
  Passing to a smooth cover of $S$, we may assume that $S$ is an
  affine scheme and so has the $\Flag$-resolution property. The result 
  follows from Theorem \ref{T:flat-local-unipotent-morphism}.
\end{proof}
\begin{corollary}\label{C:geometric-unipotent-surj-local}
  Let $f \colon X \to S$ be a morphism of algebraic
  stacks.
  \begin{enumerate}
  \item \label{CI:geometric-unipotent-surj-local:every-point}If $f$ is geometrically (R-)unipotent, then
    $f\times_S \spec k \colon X\times_S \spec k \to \spec k$ is
    (R-)unipotent for every field $k$ and morphism $\spec k \to S$.
  \item \label{CI:geometric-unipotent-surj-local:surj}  Let $S' \to S$ be surjective. If
  $f' \colon X\times_{S} S' \to S'$ is geometrically (R-)unipotent,
  then $f$ is geometrically (R-)unipotent.  
  \end{enumerate}
\end{corollary}
\begin{proof}
  In \eqref{CI:geometric-unipotent-surj-local:surj}, we may assume that $S=\Spec k$ and $S'=\Spec k'$, where
  $k \subseteq k'$ is a field extension and $k'$ is algebraically closed. It suffices to prove that
  if $f'$ is (R-)unipotent, then $f$ is (R-)unipotent, which also implies \eqref{CI:geometric-unipotent-surj-local:every-point}. But $S' \to S$
  is an fpqc covering, so $f$ is locally (R-)unipotent. Since $S$ is
  affine, it has the $\Flag$-resolution property. The result now
  follows from Theorem \ref{T:flat-local-unipotent-morphism}.
\end{proof}

\subsection{Unipotent groups}\label{ss:unipotent-groups}
A powerful source of unipotent morphisms will be unipotent groups and
gerbes. Let $k$ be a field. Let $G$ be an algebraic group over
$k$. There are several characterizations of unipotence over $k$
\cite[Thm.~XVII.3.5]{SGA3-II-NEW}. We take the following: $G$ is
\emph{unipotent} if there is an embedding into the upper triangular
unipotent matrices $G \hookrightarrow \U_{n,k} \subset \GL_{n,k}$ for
some $n$. Families of unipotent groups over general bases are more
subtle, however, so we make the following definition.
\begin{Definition}\label{D:unipotent groups}
  Let $S$ be an algebraic stack. A group $G \to S$ that is flat and of finite presentation is said to be
  \emph{(R-)unipotent, locally (R-)unipotent, or geometrically
    (R-)unipotent} if the corresponding gerbe $\classB G \to S$ is so.
\end{Definition}
\begin{remark}
  We note some potential confusion with this terminology: unipotence of
  the group $G\to S$ is not the same thing as unipotence of the
  morphism $G\to S$. For instance, every quasi-compact, quasi-separated and
  representable morphism of algebraic stacks is unipotent, but an
  affine group scheme $G\to S$ need not be unipotent in the sense of
  Definition \ref{D:unipotent groups}.
\end{remark}
\begin{remark}\label{R:unipotent-group-char}
  A group $G\to S$ that is flat and of finite presentation is unipotent (resp. R-unipotent) if and only if there exists an
  ordered sequence of line bundles $\mathscr{L}$ on $S$, an object
  $V_{\bullet}$ of $\FLAG_{\mathscr{L}}(S)$ and a monomorphism (resp.\ with quasi-affine quotient):
  \[
    G \hookrightarrow \Aut_{\FLAG_{\mathscr{L}}}(V_{\bullet}) \subset \GL(V)
  \]
  Thus, (R-)unipotent groups are (R-)embeddable. Note that the above automorphism group is an inner form of
  $\U(\mathscr{L})$. In particular, if $S$ is a scheme, then $G$ can be embedded in a Zariski form of
  $\U_{n,S}$. Hence, if $S=\Spec k$ is the spectrum of a
  field, then $G\to\Spec k$ is unipotent (in the sense of Definition
  \ref{D:unipotent groups}) if and only if there exists an embedding
  $G\subseteq\U_{n,k}$ for some $n$, so both possible definitions of a unipotent group scheme agree over a field.
\end{remark}

\begin{Example}\label{R:unipotent-group-thesame}
  By Proposition \ref{P:hom}, the group $\U_{n,S} \to S$ is
  R-unipotent, as is any flat and finitely presented closed
  subgroup $H \subset \U_{n,S}$ when either $S$ is the spectrum of a field or a Dedekind domain, or $H \to S$ is finite. In particular, if $G \to S$ is (locally) unipotent and
  either $G \to S$ is finite or $S$ is the spectrum of a field, then
  it is (locally) R-unipotent.
\end{Example}

Now translating Theorem \ref{T:flat-local-unipotent-morphism} into groups
we obtain the following.
\begin{corollary}\label{C:flat-local-strong-unipotent}
  Let $S$ be a quasi-compact algebraic stack with affine diagonal. Let $G \to S$ be a locally (R-)unipotent group.
  \begin{enumerate}
  \item If $S$ has the resolution property, then $G \to S$ is (R-)embeddable.
  \item If $S$ has the $\Flag$-resolution property, then $G \to S$ is
    (R-)unipotent.
  \end{enumerate}
\end{corollary}
Note the stark contrast to the case of
tori. Tori are always locally embeddable, but may not be embeddable even over a projective curve
\cite[X.1.6 \& XI.4.6]{SGA3-II-NEW}.

\begin{remark} \label{rem:overafieldallagree} By Example \ref{R:unipotent-group-thesame} and Corollary \ref{C:flat-local-strong-unipotent} all the definitions in Definition \ref{D:unipotent groups} agree when $S$ is the spectrum of a field. Thus, one may speculate what the best definition of unipotence should be over a general base. We argue that local R-unipotence is the best behaved as evidenced by Theorem \ref{T:flat-local-unipotent-morphism}, Examples \ref{R:unipotent-group-thesame}, \ref{E:vector-is-loc-unipotent}, \ref{E:pgroup}, \ref{E:char0-affine}, and Proposition \ref{prop:pushforward is faithful on objects}. \end{remark}

We have the following proposition, which justifies our terminology.
\begin{proposition}\label{P:stabilizer-geom-unipotent}
  Let $f \colon X \to S$ be a geometrically unipotent morphism of
  algebraic stacks. Then the relative geometric stabilizers of $f$ are all
  unipotent groups. 
\end{proposition}
\begin{proof}
  We may assume that $S=\spec l$, where $l$ is an algebraically closed
  field. Let $\bar{x}\colon \spec k \to X$ be a point of $X$, where
  $k$ is an algebraically closed field. Let
  $G_{\bar{x}} = \Aut_X(\bar{x})$ be its automorphism group. Then we
  have a representable morphism
  $\tilde{x} \colon \classB G_{\bar{x}} \to X$. By Lemma
  \ref{L:restriction-flag}\eqref{LI:restriction-flag:uni},
  $\classB G_{\bar{x}}$ is unipotent. Now apply the last sentence of Remark
  \ref{R:unipotent-group-char}. 
\end{proof}
\begin{Example}\label{E:vector-is-loc-unipotent}
  Let $S$ be an algebraic stack. Let $E$ be a locally free sheaf on $S$ of
  rank $n$. Then the vector group $\mathbf{V}(E) \to S$ is locally
  R-unipotent. We may work locally on $S$, so we may assume that
  $E\simeq \Orb_{S}^{\oplus n}$. In this case,
  $\mathbf{V}(E) \simeq \mathbf{G}_{a,S}^n$. Now
  $\mathbf{G}_{a,\Z}^n \subseteq \U_{n+1,\Z} \subseteq
  \GL_{n+1,\Z}$. Since $\Z$ is a Dedekind domain, it follows from Example
  \ref{R:unipotent-group-thesame} that
  $\mathbf{G}_{a,\Z}^n \to \spec \Z$ is R-unipotent. Hence,
  $\mathbf{V}(E) \to S$ is locally R-unipotent. If $S$ is
  quasi-compact with affine diagonal and the $\Flag$-resolution property, then
  $\mathbf{V}(E) \to S$ is even R-unipotent (Corollary
  \ref{C:flat-local-strong-unipotent}).
\end{Example}
\begin{Example}\label{E:pgroup}
  Let $p>0$ be a prime. Let $S$ be an algebraic $\F_p$-stack. Let
  $G \to S$ be a finite \'etale group scheme of degree $p^d$ for some
  $d\geq 0$. Then $G \to S$ is locally R-unipotent. Indeed, smooth
  locally on $S$ we may assume that $G \to S$ is a constant group
  scheme. We may thus assume that $S=\spec \overline{\F}_p$ and $G$ is a
  constant $p$-group. Standard facts from group theory show that $G$
  is unipotent, and the claim follows.
\end{Example}

\begin{Example} \label{E:nonseparatedline} If $X$ denotes the affine
  line with doubled origin, then it is an R-unipotent scheme. Indeed
  if $x$, $y \in X$ denote the two origins, then the trivially graded
  flag $\mathcal{O}_X=I_{\{x,y\}} \subset I_{\{x\}} \oplus I_{\{y\}}$
  is a tensor generator, where $I_C$ is the ideal sheaf associated to
  a reduced closed subscheme $C \subset X$.
\end{Example}

\subsection{Representable R-unipotent morphisms} We now briefly discuss representable (locally) R-unipotent morphisms
$f\colon X\to S$. Characterizing such morphisms seems subtle. For
instance, consider for simplicity the case when $S=\Spec k$ is the
spectrum of a field. It follows from Theorem \ref{T:char-Runi} that
representable $R$-unipotent morphisms $f\colon X\to \Spec k$ are
exactly quotients of the form $T/\U_{n}$, where $T$ is a quasi-affine
scheme over $k$ acted on freely by $\U_{n}$. In particular, a
quasi-affine morphism over $k$ is R-unipotent and hence geometrically
R-unipotent. The converse is not true, however.
\begin{Example} \label{ex:repunipotent} 
{\par\noindent}
\begin{enumerate}
\item Let $k$ be a field and denote by $X$ the complement of a single
  point in the exceptional locus of $\mathbf{A}^2_k$ blown up at the
  origin. We claim that the morphism $f\colon X \to \mathbf{A}^2$ is
  geometrically $R$-unipotent but is not quasi-affine. Indeed, the
  geometric fibers are all affine and hence the morphism is
  geometrically $R$-unipotent. On the other hand, if $f$ was
  quasi-affine then the natural map
  $g\colon X \to \Spec \Gamma(X,\mathcal{O}_X)$ would be an open immersion
   \cite[Tag \spref{01SM}]{stacks-project} and this cannot be true
  since $\Gamma(X,\mathcal{O}_X)=k[x,y]$.
\item In \cite[Ex. 3.16]{MR2335246}, Asok and Doran describe a free action
  of $\U_1=\mathbf{G}_a$ on $\mathbf{A}^5$ such that the quotient
  $\mathbf{A}^5/\mathbf{G}_a$ is not even a scheme. More generally,
  given a smooth affine scheme $T$ over $k$ equipped with a free and proper action of $\U_n$,
  \cite[Corollary 3.18]{MR2335246} yields an effective
  criterion for determining when the quotient $T/\U_n$ will be affine
  or quasi-affine.
\end{enumerate}
\end{Example}
We have the following result, whose proof we defer until \S\ref{S:faithful}.
\begin{proposition}\label{prop:unipotent + proper = finite}
  Let $f\colon X\to S$ be a representable and geometrically R-unipotent
  morphism of algebraic stacks. If $f$ is proper, then $f$ is finite.
\end{proposition}
It follows from Proposition \ref{prop:unipotent + proper = finite}
that representable geometrically R-unipotent morphisms over a field do not contain
any positive-dimensional proper subvarieties.

\subsection{Geometrically unipotent vs locally unipotent groups} \label{ss:geounivslocal} We conclude this section with some thoughts on the following natural
refinement of Question \ref{Q1}: which geometrically unipotent groups are locally unipotent?
Note that locally unipotent groups are \emph{always} quasi-affine; in
particular, they are separated and schematic. In positive characteristic, it is easy
to produce geometrically unipotent group algebraic spaces that are not
separated. There are also separated geometrically unipotent group algebraic spaces that are not schemes. In particular, these give examples of geometrically unipotent groups that are not locally unipotent.
\begin{Example}\label{E:non-separated-unipotent}   Let $k$ be a field of characteristic $2$. Let
  $S = \mathbf{A}^1_{k}$. Let $H \subseteq (\Z/2\Z)_S=G$ be the
  \'etale subgroup obtained by deleting the non-trivial point over the
  origin.
  \begin{enumerate}
  \item Let $Q=G/H \to S$ be the
    line with the doubled-origin, viewed as a group scheme. Then
    $Q \to S$ is geometrically unipotent, but non-separated. In particular, it is
    not locally unipotent.
  \item We have $H \subseteq G \subseteq \mathbf{G}_{a,S}$. Let
    $Q' = \mathbf{G}_{a,S}/H \to S$. Then $Q' \to S$ is smooth,
    geometrically unipotent, with connected fibers, non-separated, and not a scheme 
    \cite[VI$_{\mathrm{B}}$.5.5]{SGA3-I-NEW}. This shows that the closed subgroup condition in
    Proposition \ref{P:hom}\eqref{lem:i:hom:sub} is necessary.
  \item A more sophisticated example was constructed in
    \cite[X.14]{Raynaud}: take $T=\mathbf{A}^2_{k}$; then there is a
    \emph{closed} subgroup $N \subseteq \mathbf{G}_{a,T}^2$ with
    $N \to T$ \'etale. Taking $Q'' = \mathbf{G}_{a,T}^2/N$, we produce
    a smooth, separated, geometrically unipotent group of finite
    presentation with connected fibers. There it is shown that $Q''$
    is not representable by a scheme. This shows that the condition
    that the quotient is representable by a scheme in Proposition
    \ref{P:hom}\eqref{lem:i:hom:sub} is necessary.
\end{enumerate}
\end{Example}
In characteristic $0$, there are strong results in the existing
literature, however.
\begin{Example}\label{E:char0-affine} 
  Let $S$ be an algebraic stack of equicharacteristic $0$. Let
  $G \to S$ be geometrically unipotent and flat-locally
  schematic. Then $G \to S$ is affine. Indeed, this is local on $S$,
  so we may assume that $S$ is affine. By standard limit methods, we
  may assume that $S$ is of finite type over $\spec \Z$ and so
  excellent; in particular, $S$ has noetherian normalization. If $S$
  is either normal or $G \to S$ is embeddable, then
  there is an isomorphism of groups
  \[
    \exp \colon \mathbf{V}(\mathrm{Lie}(G)) \to G,
  \]
  where we endow the left hand side with the group structure coming
  from the Baker--Campbell--Hausdorff formula.  In particular, when
  $G \to S$ is commutative, it is a vector group, so is locally
  R-unipotent. The statement involving the exponential map is
  established in \cite[\S1.3]{MR3525845} and
  \cite[XV.3]{Raynaud}. That $G \to S$ is affine follows from the
  normal case and Chevalley's Theorem \cite[Thm.~8.1]{rydh-2009}.
\end{Example}
We also have the following examples.
\begin{Example}\label{E:examples-r-unipotence}
  Let $G \to S$ be an affine and geometrically unipotent group.
  \begin{enumerate}
  \item \label{EI:examples-r-unipotence:prufer} If $S=\spec A$, where $A$ is a Pr\"ufer domain, then $G \to S$
    is R-unipotent.
  \item \label{EI:examples-r-unipotence:finite} If $S$ is equicharacteristic $p>0$ and $G \to S$
  is a finite locally free commutative group scheme with dual of height $1$, then $G \to S$ is locally R-unipotent. 
\end{enumerate}
For \eqref{EI:examples-r-unipotence:prufer}: by Lemma \ref{L:pruferflag} every vector bundle on $BG$ admits the structure of a complete
$S$-graded flag. In combination with Corollary
\ref{C:res-property-dim1}, this implies that $G$ is $R$-unipotent. For \eqref{EI:examples-r-unipotence:finite}:
the exact sequence of \cite[Prop. 1.1]{AMduality} implies there is an
affine morphism $BG \to B\mathbf{V}(\omega)$ for some vector group
$\mathbf{V}(\omega) \to S$. Since $\mathbf{V}(\omega) \to S$ is
locally $R$-unipotent (Example \ref{E:vector-is-loc-unipotent}), \eqref{EI:examples-r-unipotence:finite} follows.
\end{Example}

\section{Applications}
\subsection{Unipotent gerbes and the resolution property}
Theorem
\ref{T:flat-local-unipotent-morphism} immediately yields a proof of Theorem \ref{MT:ga-gerbes-res}.
\begin{proof}[Proof of Theorem \ref{MT:ga-gerbes-res}]
  A $\mathbf{G}_a$-gerbe is locally R-unipotent (Example
  \ref{R:unipotent-group-thesame}). It follows from Theorem
  \ref{T:flat-local-unipotent-morphism} that $\ms G$ has the resolution
  property.
\end{proof}

\begin{remark} \label{rem:extendingA} Assume that $X$ lives over a
  scheme $S$ and $G \to S$ is locally R-unipotent: for example, $G=\bG_a^n$, $G=\U_{n,S}$, or $S$ is the spectrum of a field and $G \subseteq \U_{n,S}$ (see Remark \ref{R:unipotent-group-thesame}). Then the argument above for Theorem
  \ref{MT:ga-gerbes-res} also holds for $G$-gerbes $\ms G \to X$.
\end{remark}

We next obtain the following
corollary of Theorem \ref{T:flat-local-unipotent-morphism}, which is
key for our proof of Theorem \ref{MT:dm-stack-res}.
\begin{corollary}\label{C:rel-DM-p-power}
  Let $S$ be a quasi-compact algebraic stack with affine diagonal over $\mathbf{F}_p$. Let $\pi \colon \ms G \to S$ be a
  gerbe that is relatively Deligne--Mumford, separated and with
  $p$-power order inertia.
  \begin{enumerate}
  \item If $S$ has the resolution property, then $\ms G$ has the resolution property. 
  \item If $S$ has the $\Flag$-resolution property, then $\pi$ is R-unipotent. 
\end{enumerate}

\end{corollary}
\begin{proof}
  Since $\pi$ is relatively Deligne--Mumford and separated, there is a
  smooth covering $S'\to S$ such that
  $\ms G\times_S S' \simeq \classB G'$, where $G' \to S'$ is finite
  \'etale. The $p$-power order inertia assumption on $\pi$ implies
  that $G' \to S'$ has degree $p^d$ for some $d\geq 0$. Example
  \ref{E:pgroup} shows that $G' \to S'$ is locally R-unipotent. Now
  apply Theorem \ref{T:flat-local-unipotent-morphism}.
\end{proof}
\subsection{Deligne--Mumford Gerbes}
Let $X$ be an algebraic stack. Then there is an injective group homomorphism:
\[
  \mathrm{Br}(X) \to \mathrm{Br}'(X)=\H^2(X,\mathbf{G}_m)_{\text{tors}}
\]
from the geometric Brauer group of $X$ to the cohomological Brauer group of $X$. However, it is still not known when the map is surjective (see, for instance, \cite{mathur2020experiments} for an introduction to this problem).

We are now in a position to prove one of
our key results.
\begin{Theorem}\label{T:DM-gerbes}
  Let $S$ be a connected quasi-compact algebraic stack with
  affine diagonal and the resolution property over a field $k$. Let
  $\pi \colon \ms G \to S$ be a gerbe that is relatively
  Deligne--Mumford and separated. If
  $\mathrm{Br}(T)=\mathrm{Br}'(T)$ for all finite
  \'etale morphisms $T\to S$ (e.g. if $S$ is a quasi-compact scheme that admits
  an ample line bundle), then $\ms G$ has the resolution property.
\end{Theorem}
\begin{proof}
  The gerbe $\ms G \to S$ is locally banded by a finite constant group scheme
  $\underline{G}_S$. It follows that the sheaf of isomorphisms (in the
  category of Bands over $S$):
  \[
    I=\underline{\mathrm{Isom}}(\underline{\mathrm{Band}}(\ms G),\underline{G}_S)
  \]
  is a torsor under the outer automorphism group
  $\text{\underline{Out}}(\underline{G}_S) \to S$. This map is finite
  \'etale, so we may replace $S$ with $I$ as the resolution property
  descends under such maps
  \cite[Prop. 5.3(vii)]{2013arXiv1306.5418G}. Hence, we can assume
  that $\ms G \to S$ is banded by $\underline{G}_S$, which is constant.
  Now there is a factorization:
  \[
    \ms Z \to \ms G \to S,
  \]
  where $\ms Z \to S$ is a gerbe banded by the center of
  $\underline{G}_S$, which is a finite abelian constant group scheme
  and $\ms Z \to \ms G$ is finite \'etale and surjective. Indeed,
  define $\ms Z$ to be the stack of banded equivalences
  \[
    \ms Z=\underline{\mathrm{HOM}}_{\mathrm{id}}(\classB\underline{G}_S, \ms G) \to S.
  \]
  By \cite[IV.2.3.2(iii)]{giraud1966cohomologie}, $\ms Z$ is banded by
  $Z(\underline{G}_S)=\underline{Z(G)}_S$ and the natural evaluation
  map $\text{ev} \colon \ms Z \to \ms G$ induces a map on bands equal
  to the inclusion $\underline{Z(G)}_S \subset
  \underline{G}_S$. Again, by
  \cite[Prop. 5.3(vii)]{2013arXiv1306.5418G}, we may now replace
  $\ms G$ by $\ms Z$. Thus, $\ms G$ is banded by the constant group
  scheme associated to the finite abelian group
  \[
    \mathbf{Z}/p_1^{n_1}\mathbf{Z} \times \dots \times \mathbf{Z}/p_l^{n_l}\mathbf{Z},
  \]
  where the $p_i$ are primes. Hence,
  $\ms G \cong \ms R_1 \times_S \dots \times_S \ms R_l$, where
  $\ms R_i$ is a $\mathbf{Z}/p_i^{n_i}\mathbf{Z}$-gerbe for all
  $i$. For each $p_i$ that is prime to the characteristic of $k$,
  consider the finite and separable field extension $k \subseteq k'$
  obtained by adding in all the $p_i^{n_i}$-th roots of unity. As
  before, it suffices to prove the result after base changing to
  $k'$. In particular, for those $p_i$ we now have
  $\mathbf{Z}/p_i^{n_i}\mathbf{Z} \cong \mu_{p_i^{n_i}}$. Since
  $\mathrm{Br}(S)=\mathrm{Br}'(S)$, the resulting $\ms R_i$ have the
  resolution property. For the remaining $p_i$ that are the
  characteristic of $k$, these $\ms R_i$ have the resolution property
  by Corollary \ref{C:rel-DM-p-power}. An easy calculation now shows
  that $\ms G$ has the resolution property.
\end{proof}
\subsection{Deligne--Mumford stacks in positive characteristic}
We finally come to the proof of Theorem \ref{MT:dm-stack-res}.
\begin{proof}[Proof of Theorem \ref{MT:dm-stack-res}]
  We may assume that $\ms X$ is connected. By \cite[Prop. 2.1]{olssonboundedness}, the morphism $\ms X \to X$
  can be rigidified. Hence, it factors into a gerbe morphism and a
  coarse moduli space morphism:
  \[
    \ms X \to \ms X^{\text{rig}} \to X,
  \]
  where $\ms X^{\mathrm{rig}}$ has generically trivial stabilizers.
  It now follows that $\ms X^{\mathrm{rig}}$ is a global quotient
  \cite[Thm. 2.18]{MR1844577}. By \cite[Thm.\ 2.1]{MR2026412} (see also \cite[Cor.~4.5]{zbMATH07471377}), there
  is a finite flat cover $Z \to \ms X^{\mathrm{rig}}$, where $Z$ is a
  quasi-projective scheme. Base changing along this cover, we may
  replace $\ms X^{\mathrm{rig}}$ by $Z$
  \cite[Prop. 5.3(vii)]{2013arXiv1306.5418G}. Then $\ms X \to X$ is a
  separated Deligne--Mumford gerbe, where $X$ is a quasi-projective
  scheme over $\spec k$. The result now follows from Theorem
  \ref{T:DM-gerbes} and Gabber's Theorem
  \cite{deJong_result-of-Gabber}.
\end{proof}

\section{Faithful moduli spaces}\label{S:faithful}
If $V$ is a nonzero representation of the unipotent group $\U_{n,k}$
over a field $k$, then the set of fixed points $V^{\U_{n,k}}$ is
nonzero \cite[XVII, Prop.\ 3.2]{SGA3-II-NEW}. If $X$ is a quasi-affine
scheme, and $F$ is a nonzero quasi-coherent sheaf on $X$, then
$\Gamma(X,F)$ is nonzero. The following result shows that this
property generalizes to any locally R-unipotent morphism.
\begin{proposition}\label{prop:pushforward is faithful on objects}
  Let $f\colon X\to S$ be a locally R-unipotent morphism of algebraic
  stacks. If $F$ is quasi-coherent on $X$ and $f_*F=0$, then $F=0$.
\end{proposition}
\begin{proof}
  By flat base change we may assume that $f$ is R-unipotent and $S$ is
  affine. Let $N$ be a quasi-coherent $\Orb_S$-module; then
  $\Hom_{\Orb_X}(f^*N,F) = \Hom_{\Orb_S}(N,f_*F) = 0$.  It follows
  that if $V$ belongs to $\Flag_f$, then 
  $\Hom_{\Orb_X}(V,F)=0$. As $f$ is R-unipotent, it has the
  $\Flag_f$-resolution property (Theorem \ref{T:char-Runi}). Hence, $F=0$.
\end{proof} 
The following definition is due to Alper \cite{Jarod-faithful}.
\begin{definition}\label{D:faithful-moduli-space}
  Let $f\colon X \to S$ be a quasi-compact and quasi-separated
  morphism of algebraic stacks. We say that $f$ is a \emph{faithful moduli space} if:
\begin{enumerate}
\item \label{DI:faithful-moduli-space:stein} the natural map $\Orb_S \to f_*\Orb_X$ is an isomorphism (i.e., $f$ is Stein); and
\item \label{DI:faithful-moduli-space:faithful} if $F \in \QCoh(X)$ and $f_*F = 0$, then $F=0$
  (i.e., $f_*$ is zero-reflecting).
\end{enumerate}   
\end{definition}
\begin{remark}\label{R:faith-flat-bc}
  Faithful moduli spaces are compatible with quasi-affine flat base
  change. Indeed, if $p \colon S' \to S$ is flat and quasi-affine and
  $f'\colon X'=X\times_S S' \to S'$ is the induced morphism, then flat
  base change says that \eqref{DI:faithful-moduli-space:stein} is
  preserved. Let $p' \colon X'\to X$ be the induced morphism; then
  $p_*$ and $p'_*$ are zero-reflecting on quasi-coherent sheaves because
  they are quasi-affine. Hence, $p_*f'_*\simeq f_*p'_*$ is
  zero-reflecting on quasi-coherent sheaves, and so $f'_*$ is
  zero-reflecting on quasi-coherent sheaves, which proves
  \eqref{DI:faithful-moduli-space:faithful}.  In particular, if
  $f\colon X \to S$ is a faithful moduli space and $S$ has
  quasi-affine diagonal, then $f$ remains a faithful moduli space
  after arbitrary flat base change on $S$.
\end{remark}
\begin{Example}\label{E:stein-loc-R-faithful}
  A Stein and locally R-unipotent morphism is a faithful moduli space
  (Proposition \ref{prop:pushforward is faithful on objects}). Also if
  $S$ is a normal scheme, then any open immersion $U \subseteq S$ with
  complement of codimension $\geq 2$ is Stein and R-unipotent. In
  particular, there is not necessarily an induced bijection on closed
  points, so some may find the term ``moduli space'' misleading.
\end{Example}

\begin{remark}\label{rem:faithfulnotfaithful}
  Note that a faithful moduli space $f$ need not imply that the
  functor $f_*$ is faithful. Indeed, $f_*$ being faithful forces $f$
  to be quasi-affine (because the structure sheaf would be
  $f$-generating). Moreover, $f$ being Stein implies that the quasi-affine
  morphism $f$ is an open immersion. Thus, the second example in
  Example \ref{E:stein-loc-R-faithful} is essentially the only example
  of a faithful moduli space for which $f_*$ is faithful.
\end{remark}

 \begin{remark}
   The converse to Proposition \ref{prop:pushforward is faithful on
     objects} does not hold: if $X$ denotes affine $n$-space with a
   doubled origin and $f\colon X \to \mathbf{A}^n$ is the natural map,
   $f$ is a faithful moduli space but is not locally $R$-unipotent
   unless $n=1$ (see Example \ref{E:nonseparatedline}). Indeed, it
   cannot be locally $R$-unipotent when $n \geq 2$, because Theorem
   \ref{T:flat-local-unipotent-morphism} would imply that $X$ has
   affine diagonal. On the other hand, $f$ is Stein. To see that $f_*$ is
   zero-reflecting, use Lemma \ref{prop:qfinitenonsep}.
 \end{remark}

 \begin{lemma}\label{prop:qfinitenonsep}
   Let $f\colon X \to Y$ be a
   morphism of algebraic stacks which is representable and
   quasi-finite, then $f_*\colon \QCoh(X) \to \QCoh(Y)$ is
   zero-reflecting.
 \end{lemma}
 \begin{proof}
   By passing to a smooth cover of $Y$, we may assume that $Y$ is an
   affine scheme and $X$ is an algebraic space. Let $F$ be a
   quasi-coherent $\Orb_X$-module with $f_*F=0$; then the same is true
   for any of its finite-type subsheaves, so we may assume $F$ is also of
   finite-type. If $F \neq 0$, then we may replace $X$ with
   $V(\Ann_{\Orb_X}(F))$ and then $Y$ with
   $V(\ker(\Orb_Y \to f_*\Orb_X))$. Then 
   $\Ann_{\Orb_X}(F) = 0$ and $f$ has schematically dense image. Since
   $f$ is quasi-finite, there is a non-empty open subset
   $V \subseteq Y$ so that $f^{-1}(V) \to V$ is finite. Then
   $F|_{f^{-1}(V)}=0$ and so $F=0$, which is a
   contradiction.
 \end{proof}

\begin{Example}\label{E:group-faithful}
  Let $k$ be a field. Let $G$ be an algebraic group scheme over
  $k$. Then $\classB G \to \spec k$ is a faithful moduli space if and
  only if $G$ is unipotent. One implication follows from Example
  \ref{E:stein-loc-R-faithful}. For the other: if $V$ is a
  finite-dimensional representation of $G$ over $k$; then
  $V^G \neq 0$. By induction, $V$ is a trivially graded flag on
  $\classB G$. Taking $V$ to be a faithful representation of $G$, we
  see that $G$ admits a faithful flag. In particular, it is unipotent.
\end{Example}
\begin{Example}\label{E:faithful-comp}
  Let $Y \xrightarrow{g} X \xrightarrow{f} S$ be quasi-compact and
  quasi-separated morphisms of algebraic stacks.
  \begin{enumerate}
  \item\label{E:faithful-comp:comp} If $f$ and $g$ are faithful moduli spaces, then $f\circ g$ is
    a faithful moduli space.
  \item\label{E:faithful-comp:build} If
    $f_*$ is zero-reflecting, then $X \to \underline{\spec}_S(f_*\Orb_X)$ is a
    faithful moduli space.
  \end{enumerate}
\end{Example}
The following result is a version of Proposition
\ref{P:stabilizer-geom-unipotent} for faithful moduli spaces.
\begin{proposition}\label{P:res-gerb-faithful}
  Let $f \colon X \to S$ be a faithful moduli space with affine stabilizers. Then the relative
  geometric stabilizers of $f$ are unipotent groups. 
\end{proposition}
\begin{proof}
  We may assume that $S$ is an affine scheme; then $X$ is
  quasi-compact and quasi-separated. Let
  $\bar{x} \colon \spec k \to X$ be a geometric point of $X$, where
  $k$ is an algebraically closed field. Let $G= \Aut(\bar{x})$, which
  is a group scheme of finite type over $\spec k$. There is an
  induced quasi-affine morphism $\classB G \to X$
  \cite[Thm.~B.2]{MR2774654} and so $\classB G \to \spec k$ is a
  faithful moduli space (Example
  \ref{E:faithful-comp}\eqref{E:faithful-comp:build}). By Example
  \ref{E:group-faithful}, the result follows.
\end{proof}
The following lemma turns out to be very useful.
\begin{lemma}\label{L:disjoint-local}
  Let $f \colon X \to S$ be a quasi-compact, quasi-separated and Stein
  morphism of algebraic stacks. If $I \subseteq \Orb_X$ is a
  quasi-coherent sheaf of ideals, then
  $V(I) \subseteq f^{-1}(V(f_*I))$. 
\end{lemma}
\begin{proof}
  Indeed, the quasi-coherent sheaf defining the closed substack $f^{-1}(V(f_*I)) \subseteq X$ is the
  image of $f^*f_*I \to \Orb_X$, which factors through $I$. 
\end{proof}
\begin{proposition}\label{P:faithful-generic-gerbe}
  Let $f \colon X \to S$ be a faithful moduli space with affine
  stabilizers. If $S$ is reduced and quasi-separated with affine stabilizers, then there is a
  non-empty open $U \subseteq S$ with $f^{-1}(U) \to U$ a locally
  R-unipotent gerbe.
\end{proposition}
\begin{proof}
  We may assume that $S$ is quasi-compact; then
  \cite[Prop.~2.6(i)]{hallj_dary_alg_groups_classifying} implies that
  there is a non-empty open quasi-compact $U \subseteq S$ with affine
  diagonal. Let $\spec A \to U$ be a smooth cover; then the
  composition $\spec A \to U \to S$ is quasi-affine and smooth. It
  follows from Remark \ref{R:faith-flat-bc} that we may now assume
  that $S=\spec A$, where $A$ is reduced.

  Now there is a non-empty open immersion
  $\mathcal{G} \subseteq X_{\red}$ such that $\mathcal{G}$ is a gerbe
  with the resolution property with affine coarse space
  $g \colon \mathcal{G} \to T$ (combine \cite[Tags \spref{06RC} \&
  \spref{06NH}]{stacks-project} with \cite[Prop.\
  2.6(i)]{hallj_dary_alg_groups_classifying}).  Let
  $I \subseteq \Orb_X$ be a quasi-coherent ideal sheaf defining the
  reduced closed complement $Z$ of $\mathcal{G} \subseteq X$. If
  $Z\neq \emptyset$, then $\Gamma(X,\Orb_X) \to \Gamma(Z,\Orb_Z)$ is
  non-zero because it is a ring homomorphism and
  $\Gamma(Z,\Orb_Z) \neq 0$. In particular,
  $\Gamma(X,I) \subsetneq A$. Since $I\neq 0$, it follows that we may
  choose $a \in \Gamma(X,I)\setminus \{0\}$ that is not a unit. Since
  $A$ is reduced, $a$ is not nilpotent and so
  $\spec A_a \neq \emptyset$. But Lemma \ref{L:disjoint-local} implies
  that $Z \cap f^{-1}(\spec A_a)=\emptyset$ and so we may replace $A$
  by $A_a$ and assume that $Z=\emptyset$. Hence, 
  $\mathcal{G} =X_{\red}\subseteq X$ is a surjective closed immersion
  with defining ideal $J$. By Lemma \ref{L:disjoint-local} again, we
  may further shrink $S$ so that $J=0$ and $\mathcal{G} = X$ is a
  gerbe.  But the composition $\Orb_S \to t_*\Orb_T \to f_*\Orb_X$ is
  an isomorphism, as is $\Orb_T \to g_*\Orb_X$; hence, $T \simeq S$
  and $X \to S$ is of finite presentation. By standard limit methods
  and Theorem \ref{T:char-Runi}, we may replace $S$ by the spectrum of
  the localization at a minimal prime of $A$, which is a field \cite[Tag \spref{00EU}]{stacks-project}. By Example
  \ref{E:group-faithful} and 
Remark \ref{rem:overafieldallagree} the result follows.
\end{proof}

 We conclude with the following pleasant corollary.
\begin{corollary}\label{C:faithful-field}
  Let $k$ be a field. Let $X \to \spec k$ be a faithful moduli
  space. Then $X \to \spec k$ is an R-unipotent gerbe. In particular,
  if $X$ is an algebraic space, then $X \to \spec k$ is an
  isomorphism.
\end{corollary}
We conclude with the proof of our characterization of proper and
representable geometrically R-unipotent morphisms.
\begin{proof}[Proof of Proposition \ref{prop:unipotent + proper = finite}]
  We may assume that $S$ is affine. By Zariski's Main Theorem, it
  suffices to prove that $f$ is quasi-finite. Hence, we may further
  reduce to the situation where $S$ is the spectrum of an
  algebraically closed field $k$ and $f$ is R-unipotent. Let
  $X\to S'\to S$ be the Stein factorization of $f$. Then $S' \to S$ is
  finite, the morphism $X\to S'$ is again R-unipotent, and so we may
  further reduce to the situation where $X$ is geometrically
  connnected and $f$ is Stein. It follows from Example
  \ref{E:stein-loc-R-faithful} that $X \to \spec k$ is a faithful
  moduli space. By Corollary \ref{C:faithful-field}, $X \to \spec k$
  is an isomorphism, which gives the claim.
\end{proof}
\bibliographystyle{amsalpha}
\bibliography{references}
\end{document}